\documentclass{article}
\usepackage{amssymb,amsfonts,amsmath,amsthm,amsopn,amstext,amscd,latexsym,xy,stmaryrd,mathtools}
\usepackage[hidelinks]{hyperref}

\input xy
\xyoption{all}

\newtheorem{theorem}{Theorem}[section]
\newtheorem{lemma}[theorem]{Lemma}

\newtheorem{proposition}[theorem]{Proposition}
\newtheorem{corollary}[theorem]{Corollary}
\newtheorem{definition}[theorem]{Definition}

\theoremstyle{remark}

\theoremstyle{definition}
\newtheorem{construction}[theorem]{Construction}

\DeclareMathOperator{\GL}{GL}

\DeclareMathOperator{\Def}{Def}

\DeclareMathOperator{\Frac}{Frac}
\DeclareMathOperator{\Frob}{Frob}

\DeclareMathOperator{\Hom}{Hom}

\DeclareMathOperator{\Fil}{Fil}

\DeclareMathOperator{\Gal}{Gal}
\DeclareMathOperator{\Aut}{Aut}

\DeclareMathOperator{\ad}{ad}

\DeclareMathOperator{\End}{End}
\DeclareMathOperator{\tr}{tr}

\DeclareMathOperator{\Res}{Res}

\DeclareMathOperator{\Art}{Art}

\DeclareMathOperator{\Iw}{Iw}

\newcommand{\cA}{{\mathcal A}}

\newcommand{\cC}{{\mathcal C}}
\newcommand{\cD}{{\mathcal D}}
\newcommand{\cE}{{\mathcal E}}

\newcommand{\cG}{{\mathcal G}}

\newcommand{\cJ}{{\mathcal J}}

\newcommand{\cM}{{\mathcal M}}

\newcommand{\cO}{{\mathcal O}}
\newcommand{\cP}{{\mathcal P}}

\newcommand{\cS}{{\mathcal S}}

\newcommand{\frl}{{\mathfrak l}}
\newcommand{\ffrm}{{\mathfrak m}}

\newcommand{\frp}{{\mathfrak p}}
\newcommand{\frq}{{\mathfrak q}}

\newcommand{\bbA}{{\mathbb A}}

\newcommand{\bbC}{{\mathbb C}}

\newcommand{\bbQ}{{\mathbb Q}}
\newcommand{\bbR}{{\mathbb R}}

\newcommand{\bbT}{{\mathbb T}}

\newcommand{\bbZ}{{\mathbb Z}}

\newcommand{\wv}{{\widetilde{v}}}

\newcommand{\barQl}{{\overline{\bbQ}_l}}

\newcommand{\num}{d}

\title{Automorphy lifting for residually reducible $l$-adic Galois representations, II}
\begin{document}
	\author{Patrick B. Allen, James Newton, and Jack A. Thorne}
	
\maketitle
\begin{abstract}
	We revisit the paper \cite{Tho16} by the third author, proving new automorphy lifting theorems for residually reducible Galois representations of unitary type in which the residual representation is permitted to have an arbitrary number of irreducible constituents. 
\end{abstract}
\tableofcontents
\section{Introduction}

In this paper, we prove new automorphy lifting theorems for Galois representations of unitary type. Thus we are considering representations $\rho : G_F \to \GL_n(\overline{\bbQ}_l)$, where $G_F$ is the absolute Galois group of a CM field $F$, and $\rho$ is conjugate self-dual, i.e.\ there is an isomorphism $\rho^c \cong \rho^\vee \otimes \epsilon^{1-n},$ where $c \in \Aut(F)$ is complex conjugation. We say in this paper that such a representation is automorphic if there exists a regular algebraic, conjugate self-dual, cuspidal automorphic representation $\pi$ which is matched with $\rho$ under the Langlands correspondence. (See \S \ref{subsec_notation} below for a more precise formulation.)

We revisit the context of the paper \cite{Tho16}, proving theorems valid in the 
case that $\overline{\rho}$ is absolutely reducible, but still satisfies a 
certain non-degeneracy condition (we say that $\overline{\rho}$ is ``Schur''). 
The first theorems of this type were proved in the paper \cite{Tho16}, under 
the assumption that $\overline{\rho}$ has only two irreducible constituents. 
Our main motivation here is to remove this restriction. Our results are applied 
to the problem of symmetric power functoriality in \cite{new-tho-symm}, where 
they are combined with level raising theorems to establish automorphy of 
symmetric powers for certain level $1$ Hecke eigenforms congruent to a theta 
series.

We are also able to weaken some other hypotheses in \cite{Tho16}, leading to 
the following 
result, which is the main theorem of this paper.
\begin{theorem}[Theorem \ref{thm_new_main_theorem}]\label{thm_intro_theorem}
	 Let $F$ be an imaginary CM number field with maximal totally real subfield $F^+$, and let $n \geq 2$ be an integer. Let $l$ be a prime, and suppose that $\rho : G_F \rightarrow \mathrm{GL}_n(\overline{\bbQ}_l)$ is a continuous semisimple representation satisfying the following hypotheses.
	\begin{enumerate} \item $\rho^c \cong \rho^\vee \epsilon^{1-n}$.
		\item $\rho$ is ramified at only finitely many places.
		\item  $\rho$ is ordinary of weight $\lambda$, for some $\lambda \in (\bbZ_+^n)^{\Hom(F, \barQl)}$.
		\item There is an isomorphism $\overline{\rho}^{\text{ss}} \cong 
		\overline{\rho}_1 \oplus \dots \oplus \overline{\rho}_\num$, where each 
		$\overline{\rho}_i$ is absolutely irreducible and satisfies 
		$\overline{\rho}_i^c \cong \overline{\rho}_i^\vee \epsilon^{1-n}$, and 
		$\overline{\rho}_i \not\cong \overline{\rho}_j$ if $i \neq j$.
		\item There exists a finite place $\wv_0$ of $F$, prime to $l$, such 
		that $\rho|_{G_{F_{\wv_0}}}^\text{ss} \cong \oplus_{i=1}^n \psi 
		\epsilon^{n-i}$ for some unramified character $\psi : G_{F_{\wv_0}} 
		\rightarrow \overline{\bbQ}_l^\times$.
		\item There exists a RACSDC representation $\pi$ of $\mathrm{GL}_n(\bbA_F)$ and $\iota : \barQl \to \bbC$ such that: \begin{enumerate} \item $\pi$ is $\iota$-ordinary. 
			\item $\overline{r_{ \iota}(\pi)}^\text{ss} \cong \overline{\rho}^\text{ss}$.
			\item $\pi_{\wv_0}$ is an unramified twist of the Steinberg representation.
		\end{enumerate} 
		\item $F(\zeta_l)$ is not contained in $\overline{F}^{\ker \ad 
	(\overline{\rho}^\text{ss})}$ and $F$ is not contained in $F^+(\zeta_l)$. For 
each $1 \leq i, j \leq \num$, $\overline{\rho}_i|_{G_{F(\zeta_l)}}$ is 
absolutely 
irreducible and $\overline{\rho}_i|_{G_{F(\zeta_l)}} \not\cong 
\overline{\rho}_j|_{G_{F(\zeta_l)}}$ if $i \neq j$. Moreover, 
$\overline{\rho}^\text{ss}$ is primitive (i.e.~not induced from any proper 
subgroup of $G_F$) and $\overline{\rho}^\text{ss}(G_{F})$ has no quotient of 
order $l$.
		\item $l > 3$ and $l \nmid n$.
	\end{enumerate}
	Then $\rho$ is automorphic: there exists an $\iota$-ordinary RACSDC automorphic representation $\Pi$ of $\GL_n(\bbA_F)$ such that $r_\iota(\Pi) \cong \rho$. 
\end{theorem}
Comparing this with \cite[Theorem 7.1]{Tho16}, we see that we now allow an arbitrary number of irreducible constituents, while also removing the requirement that the individual constituents are adequate (in the sense of \cite{Tho11}) and potentially automorphic. This assumption of potential automorphy was used in \cite{Tho16}, together with the Khare--Wintenberger method, to get a handle on the quotient of the universal deformation ring of $\overline{\rho}$ corresponding to reducible deformations. This made generalising \cite[Theorem 7.1]{Tho16} to the case where more than two irreducible constituents are allowed seem a formidable task: one would want to know that any given direct sum of irreducible constituents of $\overline{\rho}$ was potentially automorphic, and then perhaps use induction on the number of constituents to control the reducible locus.

The first main innovation in this paper that allows us to bypass this is the observation that by fully exploiting the `connectedness dimension' argument to prove $R = \bbT$ (which goes back to \cite{Ski99}, and appears in this paper in the proof of Theorem \ref{thm_main_argument}), one only needs to control the size of the reducible locus in quotients of the universal deformation ring which are known a priori to be finite over the Iwasawa algebra $\Lambda$. This can be done easily by hand using the `locally Steinberg' condition (as in \S \ref{subsec_dimension_bounds}). 

The second main innovation is a finer study of the universal deformation ring 
$R^{\text{univ}}$ of a (reducible but) Schur residual representation. We show 
that if the residual representation has $\num$ absolutely irreducible 
constituents, then there is an action of a group $\mu_2^\num$ on 
$R^{\text{univ}}$, and identify the invariant subring 
$(R^{\text{univ}})^{\mu_2^\num}$ with the subring topologically generated by 
the 
traces of Frobenius elements (which can also be characterised as the image $P$ 
of the canonical map to $R^{\text{univ}}$ from the universal pseudodeformation 
ring). This leads to a neat proof that the map $P \to R^{\text{univ}}$ is 
\'etale at prime ideals corresponding to irreducible deformations of 
$\overline{\rho}$. 

We now describe the organisation of this paper. Since it is naturally a continuation of \cite{Tho16}, we maintain the same notation and use several results and constructions from that paper as black boxes. We begin in \S \ref{sec_determinants} and \ref{sec_deformation_theory} by extending several results from \cite{Tho16} about the relation between deformations and pseudodeformations to the case where $\overline{\rho}$ is permitted to have more than 2 irreducible constituents. We also make the above-mentioned study of the dimension of the locus of reducible deformations.

In \S \ref{sec_automorphic_forms} we recall from \cite{Tho16} the definition of 
the unitary group automorphic forms and Hecke algebras that we use, and state 
the $\bbT_\frq = R_\frp$ type result proved in that paper (here $\frp$ denotes 
a dimension 1, characteristic $l$ prime of $R$ with good properties, in 
particular that the associated representation to $\GL_n(\Frac R / \frp)$ is 
absolutely irreducible). In \S \ref{sec_mainargument} we carry out the main 
argument, based on the notion of connectedness dimension, which is described 
above. Finally, in \S \ref{sec_main_theorem} we deduce Theorem 
\ref{thm_intro_theorem}, following a simplified version of the argument in 
\cite[\S 7]{Tho16} that no longer makes reference to potential automorphy.

\subsection*{Acknowledgments}

We would like to thank the organizers of the workshop Deformation Theory, 
Completed Cohomology, Leopoldt Conjecture and K-theory at CIRM in December 
2016, where some of the ideas of this paper were first discussed. We also thank 
the anonymous referee for their helpful comments.

This work was partially conducted during the period that J.T. served as a Clay 
Research Fellow. J.T.'s work received funding from the European Research 
Council (ERC) under the European Union's Horizon 2020 research and innovation 
programme (grant agreement No 714405). J.T. would like to thank Lue Pan for 
pointing out an error in \cite{Tho16}, which is addressed here in \S 
\ref{subsec_erratum} below.

P.A. was supported by Simons Foundation Collaboration Grant 527275 and NSF grant DMS-1902155.

\subsection{Notation}\label{subsec_notation}

We use the same notation and normalizations for Galois groups, class field theory, and local Langlands correspondences as in \cite[Notation]{Tho16}. Rather than repeat this verbatim here we invite the reader to refer to that paper for more details. We do note the convention that if $R$ is a ring and $P$ is a prime ideal of $R$, then $R_{(P)}$ denotes the localisation of $R$ at $P$ and $R_P$ denotes the completion of the localisation. 

We recall that $\bbZ_n^+ \subset \bbZ^n$ denote the set of tuples $\lambda = (\lambda_1, \dots, \lambda_n)$ of integers such that $\lambda_1 \geq \dots \geq \lambda_n$. It is identified in a standard way with the set of highest weights of $\GL_n$. If $F$ is a number field and $\lambda = (\lambda_\tau) \in (\bbZ_n^+)^{\Hom(F, \bbC)}$, then we write $\Xi_\lambda$ for the algebraic representation of $\GL_n(F \otimes_\bbQ \bbC) = \prod_{\tau \in \Hom(F, \bbC)} \GL_n(\bbC)$ of highest weight $\lambda$. If $\pi$ is an automorphic representation of $\GL_n(\bbA_F)$, we say that $\pi$ is regular algebraic of weight $\lambda$ if $\pi_\infty$ has the same infinitesimal character as the dual $\Xi_\lambda^\vee$.

Let $F$ be a CM field (i.e.\ a totally imaginary quadratic extension of a 
totally real field $F^+$). We always write $c \in \Aut(F)$ for complex 
conjugation. We say that an automorphic representation $\pi$ of $\GL_n(\bbA_F)$ 
is conjugate self-dual if there is an isomorphism $\pi^c \cong \pi^\vee$. We 
use the acronym RACSDC to denote regular algebraic, conjugate self-dual, 
cuspidal. If $\pi$ is a RACSDC automorphic representation of $\GL_n(\bbA_F)$, 
and $\iota : \overline{\bbQ}_l \to \bbC$ is an isomorphism (for some prime 
$l$), then there exists an associated Galois representation $r_\iota(\pi) : G_F 
\to \GL_n(\overline{\bbQ}_l)$, characterised up to isomorphism by the 
requirement of compatibility with the local Langlands correspondence at each 
finite place of $F$; see \cite[Theorem 2.2]{Tho16} for a reference. We say that 
a representation $\rho : G_F \to \GL_n(\overline{\bbQ}_l)$ is automorphic if 
there exists a choice of $\iota$ and RACSDC $\pi$ such that $\rho \cong 
r_\iota(\pi)$.

One can define what it means for a RACSDC automorphic representation $\pi$ to be $\iota$-ordinary (see \cite[Lemma 2.3]{Tho16}; it means that the eigenvalues of certain Hecke operators, a priori $l$-adic integers, are in fact $l$-adic units). If $\mu \in (\bbZ^n_+)^{\Hom(F, \overline{\bbQ}_l)}$, we say (following \cite[Definition 2.5]{Tho16}) that a representation $\rho : G_F \to \GL_n(\overline{\bbQ}_l)$ is ordinary of weight $\mu$ if for each place $v | l$ of $F$, there is an isomorphism
\[ \rho|_{G_{F_v}} \sim  \left(\begin{array}{cccc} \psi_1 & \ast & \ast & \ast \\
0 & \psi_2 & \ast & \ast \\
\vdots & \ddots & \ddots & \ast \\
0 & \dots & 0 & \psi_n \end{array}\right), \]
where $\psi_i : G_{F_v} \rightarrow \barQl^\times$ is a continuous character satisfying the identity
\[ \psi_i(\sigma) = \prod_{\tau : F_v \hookrightarrow \barQl} \tau(\Art_{F_v}^{-1}(\sigma))^{-( \mu_{\tau, n - i + 1} + i - 1)} \]
for all $\sigma$ in a suitable open subgroup of $I_{F_v}$. An important result (\cite[Theorem 2.4]{Tho16}) is that if $\pi$ is RACSDC of weight $\lambda$ and $\iota$-ordinary, then $r_\iota(\pi)$ is ordinary of weight $\iota \lambda$, where by definition $(\iota \lambda)_\tau = \lambda_{\iota \tau}$.

\section{Determinants}\label{sec_determinants}

We first give the definition of a determinant from \cite{Che11}. We recall that 
if $A$ is a ring and $M, N$ are $A$-modules, then an $A$-polynomial law $F : M 
\to N$ is a natural transformation $F : h_M \to h_N$, where $h_M : A$-alg $\to$ 
Sets is the functor $h_M(B) = M \otimes_A B$. The $A$-polynomial law $F$ is 
called homogeneous of degree $n \geq 1$ if for all $b \in B$, $x \in M 
\otimes_A B$, we have $F_B(bx) = b^n F_B(x)$.
\begin{definition}
	Let $A$ be a ring and let $R$ be an $A$-algebra. An $A$-valued determinant 
	of $R$ of dimension $n \geq 1$ is a multiplicative $A$-polynomial law $D : 
	R \to A$ which is homogeneous of degree $n$. 
\end{definition}
If $D$ is a determinant, then there are associated polynomial laws $\Lambda_i : R \to A$, $i = 0, \dots, n$, given by the formulae
\[ D(t - r) = \sum_{i=0}^n (-1)^i  \Lambda_i(r)  t^{n-i} \]
for all $r \in R \otimes_A B$. We define the characteristic polynomial $A$-polynomial law $\chi : R \to R$ by the formula $\chi(r) = \sum_{i=0}^n (-1)^i  \Lambda_i(r) r^{n-i}$ ($r \in R \otimes_A B$). We write $\operatorname{CH}(D)$ for the two-sided ideal of $R$ generated by the coefficients of $\chi(r_1 t_1 + \cdots + r_m t_m) \in R[t_1, \ldots, t_m]$, for all $m \ge 1$ and $r_1,\ldots,r_m \in R$. We have $\operatorname{CH}(D) \subseteq \ker(D)$ (\cite[Lemma~1.21]{Che11}). The determinant $D$ is said to be Cayley--Hamilton if $\operatorname{CH}(D) = 0$, equivalently if $\chi = 0$ (i.e. $\chi$ is the zero $A$-polynomial law). 

We next recall the definition of a generalized matrix algebra \cite[Definition~1.3.1]{BCAst}.
\begin{definition}\label{def:GMA}
Let $A$ be a ring and let $R$ be an $A$-algebra. 
We say $R$ is a \emph{generalized matrix algebra} of type $(n_1,\ldots,n_\num)$ 
if it is equipped with 
the following data:
\begin{enumerate}
\item a family of orthogonal idempotents $e_1,\ldots, e_\num$ with $e_1+\cdots 
+e_\num = 1$, and
\item for each $1\le i \le \num$, an $A$-algebra isomorphism $\psi_i \colon e_i 
R e_i \rightarrow M_{n_i}(A)$,
\end{enumerate}
such that the trace map $T \colon R \rightarrow A$ defined by $T(x) = 
\sum_{i=1}^\num \tr\psi_i(e_ix e_i)$ satisfies $T(xy) = T(yx)$ for all $x,y \in 
R$.
We refer to the data $\mathcal{E} = \{e_i,\psi_i, 1\le i \le \num\}$ as the 
\emph{data of idempotents} of $R$.
\end{definition}

\begin{construction}\label{GMAconstruction}
We recall the structure of generalized matrix algebras from \cite[\S1.3.2]{BCAst}. 
Let $R$ be a generalized matrix algebra of type $(n_1,\ldots,n_\num)$ with data 
of idempotents $\mathcal{E} = \{e_i,\psi_i, 1\le i \le \num\}$. 
For each $1\le i \le \num$, let $E_i \in e_i R e_i$ be the unique element such 
that $\psi_i(E_i)$ is the element of $M_{n_i}(A)$ whose row $1$, column $1$ 
entry is $1$, and all other entries are $0$. 
We set $\mathcal{A}_{i,j} = E_i R E_j$ for each $1\le i,j\le \num$. 
Note that $\mathcal{A}_{i,j}\mathcal{A}_{j,k} \subseteq \mathcal{A}_{i,k}$ for 
each $1\le i,j,k\le \num$, and the trace map $T$ induces an isomorphism 
$\mathcal{A}_{i,i} \cong A$ for each $1\le i \le \num$.
Via this isomorphism, we will tacitly view $\mathcal{A}_{i,j}\mathcal{A}_{j,i}$ 
as an ideal in $A$ for each $1\le i,j\le \num$.
With this multiplication, there is an isomorphism of $A$-algebras
	\begin{equation}\label{eqn:GMA}
    R \cong \begin{pmatrix}
    M_{n_1}(A) & M_{n_1,n_2}(\mathcal{A}_{1,2}) & \cdots & 
    M_{n_1,n_{\num}}(\mathcal{A}_{1,{\num}})\\
    M_{n_2,n_1}(\mathcal{A}_{2,1}) & M_{n_2}(A) & \cdots & 
    M_{n_2,n_{\num}}(\mathcal{A}_{2,{\num}})\\
    \vdots & \vdots & \ddots & \vdots \\
    M_{n_{\num},n_1}(\mathcal{A}_{{\num},1}) & 
    M_{n_{\num},n_2}(\mathcal{A}_{{\num},2}) & \cdots & M_{n_{\num}}(A)
    \end{pmatrix}.
    \end{equation}
\end{construction}
The following result of Chenevier allows us to use the above structure when studying determinants.
\begin{theorem}\label{thm:detGMA}
Let $A$ be a Henselian local ring with residue field $k$, let $R$ be an 
$A$-algebra, and let $D \colon R \rightarrow A$ be a Cayley--Hamilton 
determinant. Suppose there exist surjective and pair-wise non-conjugate 
$k$-algebra homomorphisms $\overline{\rho}_i : R \to M_{n_i}(k)$ such that 
$\overline{D} = \prod_{i=1}^\num (\det \circ \overline{\rho}_i)$, where 
$\overline{D} = D \otimes_R k$. 

Then there is a data of idempotents $\mathcal{E} = \{e_i,\psi_i, 1\le i \le 
\num\}$ for which $R$ is a generalized matrix algebra and such that $\psi_i 
\otimes_A k = \overline{\rho}_i|_{e_i R e_i}$. 
Any two such data are conjugate by an element of $R^\times$.
\end{theorem}
We note that  the assumptions of Theorem \ref{thm:detGMA} say that $D$ is residually split and multiplicity-free, in the sense of \cite[Definition 2.19]{Che11}.
\begin{proof}
The existence of such a data of idempotents $\mathcal{E} = \{e_i,\psi_i,1\le i 
\le \num\}$ is contained in \cite[Theorem~2.22]{Che11} and its proof. 
The statement that two such data are conjugate is exactly as in \cite[Lemma~1.4.3]{BCAst}. 
Namely, if $\mathcal{E}' = \{e_i',\psi_i',1\le i \le \num\}$ is another such 
choice, 
then since $\End_R(R e_i) \cong M_{n_i}(A) \cong \End_R(R e_i')$ are local 
rings, the Krull--Schmidt--Azumaya Theorem 
\cite[Theorem~6.12]{curtis-reiner-methods} (see also \cite[Remark~6.14 and 
Chapter 6, Exercise 14]{curtis-reiner-methods}) implies there is $x\in 
R^\times$ such that $x e_i x^{-1} = e_i'$ for each $1\le i \le \num$. 
By Skolem--Noether, we can adjust $x$ by an element of $(\oplus_{i=1}^\num e_i 
R e_i)^\times$ 
so that it further satisfies $x\psi_i x^{-1} = \psi_i'$.
\end{proof}

We now show that the reducibility ideals of \cite[Proposition~1.5.1]{BCAst} and their basic properties carry over for determinants (so without having to assume that $n!$ is invertible in $A$).

\begin{proposition}\label{prop:reducibleideal}
Let $A$ be a Henselian local ring with residue field $k$, let $R$ be an $A$-algebra, and let $D \colon R \rightarrow A$ be a determinant. 
Assume that $\overline{D} = D\otimes_A k \colon R\otimes_A k \rightarrow k$ is split and multiplicity free. 
Write $\overline{D} = \prod_{i=1}^\num \overline{D}_i$, with each 
$\overline{D}_i$ absolutely irreducible of dimension $n_i$.

Let $\mathcal{P} = (\mathcal{P}_1,\ldots,\mathcal{P}_s)$ be a partition of 
$\{1,\ldots,\num\}$. 
There is an ideal $I_\mathcal{P}$ of $A$ such that an ideal $J$ of $A$ satisfies $I_\mathcal{P} \subseteq J$ if and only if there are determinants $D_1,\ldots,D_s \colon R\otimes_A A/J \rightarrow A/J$ such that $D\otimes_A A/J = \prod_{m=1}^s D_m$ and $D_m \otimes_A k = \prod_{i \in \mathcal{P}_m} \overline{D}_i$ for each $1\le m \le s$.
If this property holds, then $D_1,\ldots,D_s$ are uniquely determined and 
satisfy $\ker (D\otimes_A A/J) \subseteq \ker (D_m)$.

Moreover, let $\mathcal{J}$ be a two sided ideal of $R$ with $\operatorname{CH}(D) \subseteq \mathcal{J} \subseteq \ker(D)$ and let $\mathcal{A}_{i,j}$ be the $A$-modules as in Construction~\ref{GMAconstruction} for a choice of data of idempotents as in Theorem~\ref{thm:detGMA} applied to $R/\mathcal{J}$.
Then $I_\mathcal{P} = \sum_{i,j} \mathcal{A}_{i,j}\mathcal{A}_{j,i}$ where the sum is over all pairs $i,j$ not belonging to the same $\mathcal{P}_m\in \mathcal{P}$.
\end{proposition}

\begin{proof}
We follow the proof of \cite[Proposition~1.5.1]{BCAst} closely.
Choose a two sided ideal $\mathcal{J}$ of $R$ with $\operatorname{CH}(D) \subseteq \mathcal{J} \subseteq \ker(D)$, and data of idempotents $\mathcal{E}$ for $R/\mathcal{J}$ as in Theorem~\ref{thm:detGMA}. 
We let $\mathcal{A}_{i,j}$ be as in Construction~\ref{GMAconstruction}, and define $I_\mathcal{P} = \sum_{i,j} \mathcal{A}_{i,j}\mathcal{A}_{j,i}$ where the sum is over all pairs $i,j$ not belonging to the same $\mathcal{P}_m\in \mathcal{P}$.
Since another such choice of the data of idempotents is conjugate by an element of $(R/\mathcal{J})^\times$, the ideal $I_{\mathcal{P}}$ does not depend on the choice of $\mathcal{E}$. 
To see that it is independent of $\mathcal{J}$, first note that $D$ further factors through a surjection $\psi \colon R/\mathcal{J} \rightarrow R/\ker(D)$.
Under this surjection, the data of idempotents $\mathcal{E}$ is sent to a data of idempotents for $R/\ker(D)$, and $\tr(\psi(\mathcal{A}_{i,j})\psi(\mathcal{A}_{j,i})) = \tr(\mathcal{A}_{i,j}\mathcal{A}_{j,i})$ since $\tr\circ \psi = \tr$. 

We can now replace $R$ with $R/\operatorname{CH}(D)$ and assume that $D$ is Cayley--Hamilton. 
Since $\operatorname{CH}(D)$ is stable under base change, it suffices to show that $I_{\mathcal{P}} = 0$ if and only if there are determinants $D_1,\ldots,D_s \colon R\rightarrow A$ such that $D = \prod_{m=1}^s D_m$ and $D_m \otimes_A k_A = \prod_{i \in \mathcal{P}_m} \overline{D}_i$ for each $1\le m \le s$, and that if this happens, then $D_1,\ldots,D_s$ are uniquely determined.
Fix a data of idempotents $\mathcal{E} = \{e_i,\psi_i,1\le i \le \num\}$ for 
$R$ as in Theorem~\ref{thm:detGMA}, and let the notation be as in 
Construction~\ref{GMAconstruction}. 
For each $1\le m \le s$, we set $f_m = \sum_{i\in \mathcal{P}_m} e_i$.
Then $1 = f_1+\cdots + f_s$ is a decomposition into orthogonal idempotents. 

First assume that $I_{\mathcal{P}} = 0$. 
Let $\widetilde{D}$ denote the $A$-valued determinant on $R/\ker(D)$ arising from $D$. 
Fix $x \in R$, an $A$-algebra $B$, and $y \in R\otimes_A B$. 
If $1\le i,j\le \num$ do not belong to the same $\mathcal{P}_m\in \mathcal{P}$,
then using the algebra structure as in \eqref{eqn:GMA} and the fact that $\mathcal{A}_{i,j} \mathcal{A}_{j,i} = 0$, we have $e_i x e_j y = \sum_{l\ne i} e_i x e_j y e_l$, and \cite[Lemma~1.12(i)]{Che11} gives
	\[
    D(1 + e_i x e_ j y) = D(1+\sum_{l\ne i} e_i x e_j y e_l) = D(1+\sum_{l\ne i} x e_j y e_l e_i) = D(1) = 1.
    \]
By \cite[Lemma~1.19]{Che11}, $e_i x e_j \in \ker(D)$ for all $x\in R$ and all $i,j$ that do not belong to the same $\mathcal{P}_m\in \mathcal{P}$.
We then have an isomorphism of $A$-algebras $R/\ker(D) \cong \prod_{m=1}^s f_m 
(R/\ker(D)) f_m$ and \cite[Lemma~2.4]{Che11} gives $D = \prod_{m=1}^s D_m$, 
where $D_m \colon R \rightarrow A$ is the composite of the surjection $R 
\rightarrow f_m (R/\ker(D)) f_m$ with the determinant $\widetilde{D}_m \colon 
f_m (R/\ker(D)) f_m \rightarrow A$ given by $x\mapsto \widetilde{D}(x + 
1-f_m)$. 
It is immediate that $D_m \otimes_A k_A = \prod_{i \in \mathcal{P}_m} \overline{D}_i$ for each $1\le m \le s$.

Now assume that there are determinants $D_1,\ldots,D_s \colon R\rightarrow A$ such that $D = \prod_{m=1}^s D_m$ and $D_m \otimes_A k = \prod_{i \in \mathcal{P}_m} \overline{D}_i$ for each $1\le m \le s$. 
The determinants $D_m$ have dimension $d_m := \sum_{i\in \mathcal{P}_m} n_i$.
The trace map yields an equality
	\[
    \sum_{1\le m\ne m' \le s} \tr(f_m R f_{m'}R f_m) = I_{\mathcal{P}}.
    \]
So to show $I_{\mathcal{P}} = 0$, it suffices to show that $\tr(f_m R f_{m'} R f_m) = 0$ for $m\ne m'$.
For this, it suffices to show that $f_{m'} \in \ker(D_m)$ for any $m\ne m'$, since this implies that $f_m R f_{m'} \in \ker(D_l)$ for any $1\le l\le s$, hence
    \[
     D(1+tf_m R f_{m'} R f_m) = \prod_{l=1}^s D_l(1+tf_m R f_{m'} R f_m) = 1.
    \]

For any idempotent $f$ of $R$, we have the determinant $D_{m,f} \colon f R f \rightarrow A$ given by $D_{m,f}(x) = D_m(x+1-f)$.
When $f = f_m$,
	\[
    D_{m,f_m} \otimes_A k = \prod_{i\in \mathcal{P}_m} \overline{D}_{i,f_m} = \prod_{i\in \mathcal{P}_m}\overline{D}_{i,e_i}
    \]
has dimension $d_m$. 
Then \cite[Lemma~2.4(2)]{Che11} implies that $D_{m,1-f_m}$ has dimension $0$, i.e. is constant and equal to $1$.
So for any $m\ne m'$, the characteristic polynomial of $f_{m'}$ with respect to $D_m$ is
	\[
    D_m(t-f_{m'}) = D_{m,f_m}(t)D_{m,1-f_m}(t-f_{m'}) = t^{d_m}.
    \]
Then $f_{m'} = f_{m'}^{d_m} \in \operatorname{CH}(D_m) \subseteq \ker(D_m)$, which is what we wanted to prove.
This further shows that for each $1\le m \le s$, 
the determinant $D_m$ is the composite of the surjections
	\[
    R \rightarrow \oplus_{l=1}^s f_l R f_l \rightarrow f_m R f_m
    \]
with the determinant $D_{f_m} \colon f_m R f_m \rightarrow A$.
Since any two choices of the data of idempotents are conjugate under $R^\times$,
each $D_m$ is uniquely determined by $D$.
\end{proof}

\section{Deformations}\label{sec_deformation_theory}

Galois deformation theory plays an essential role in this paper. The set of results we use is essentially identical to that of \cite{Tho16}, with some technical improvements. In this section we recall the notation used in \cite{Tho16}, without giving detailed definitions or proofs; we then proceed to prove the new results that we need. Some of the definitions recalled here were first given in \cite{Clo08} or \cite{Ger09}, but in order to avoid sending the reader to too many different places we restrict our citations to \cite{Tho16}.

We will use exactly the same set-up and notation for deformation theory as in \cite{Tho16}. We recall that this means that we fix at the outset the following objects:
\begin{itemize}
\item A CM number field $F$, with its totally real subfield $F^+$.
\item An odd prime $l$ such that each $l$-adic place of $F^+$ splits in $F$. We write $S_l$ for the set of $l$-adic places of $F^+$.
\item A finite set $S$ of finite places of $F^+$ which split in $F$. We assume that $S_l \subset S$ and write $F(S)$ for the maximal extension of $F$ which is unramified outside $S$ and set $G_{F, S} = \Gal(F(S) / F)$ and $G_{F^+, S} = \Gal(F(S) / F^+)$. We fix a choice of complex conjugation $c \in G_{F^+, S}$.
\item For each $v \in S$ we fix a choice of place $\wv$ of $F$ such that $\wv|_{F^+} = v$, and define $\widetilde{S} = \{ \wv \mid v \in S \}$.
\end{itemize}
We also fix the following data:
\begin{itemize}
\item A coefficient field $K \subset \overline{\bbQ}_l$ with ring of integers $\cO$, residue field $k$, and maximal ideal $\lambda \subset \cO$.
\item A continuous homomorphism $\chi : G_{F^+, S} \to \cO^\times$. We write $\overline{\chi} = \chi \text{ mod } \lambda$.
\item A continuous homomorphism $\overline{r} : G_{F^+, S} \to \cG_n(k)$ such that $\overline{r}^{-1}(\cG_n^\circ(k)) = G_{F, S}$. Here $\cG_n$ is the algebraic group over $\bbZ$ defined in \cite[\S 2.1]{Clo08}. We follow the convention that if $R : \Gamma \to \cG_n(A)$ is a homomorphism and $\Delta \subset \Gamma$ is a subgroup such that $R(\Delta) \subset \cG_n^0(A)$, then $R|_\Delta$ denotes the composite homomorphism
\[ \Delta \to \cG_n^0(A) = \GL_n(A) \times \GL_1(A) \to \GL_n(A). \]
Thus $\overline{r}|_{G_{F, S}}$ takes values in $\GL_n(k)$.
\end{itemize}
If $v \in S_l$, then we write $\Lambda_v = \cO \llbracket 
(I^\text{ab}_{F_\wv}(l))^n \rrbracket$, where $I^\text{ab}_{F_\wv}(l)$ denotes the inertia group in 
the maximal abelian pro-$l$ extension of $F_\wv$. 
We set $\Lambda = \widehat{\otimes}_v \Lambda_v$, the completed tensor product 
being over $\cO$. A \emph{global deformation problem}, as defined in \cite[\S 
3]{Tho16}, then consists of a tuple
\[ \cS = (  F / F^+, S, \widetilde{S},  \Lambda, \overline{r}, \chi, \{ \cD_v \}_{v \in S} ). \]
The extra data that we have not defined consists of the choice of a \emph{local deformation problem} $\cD_v$ for each $v \in S$. We will not need to define any new local deformation problems in this paper, but we recall that the following have been defined in \cite{Tho16}:
\begin{itemize}
\item ``Ordinary deformations'' give rise to a problem $\cD_v^\triangle$ for each $v \in S_l$ (\cite[\S 3.3.2]{Tho16});
\item ``Steinberg deformations'' give rise to a problem $\cD_v^\text{St}$ for 
each place $v \in S$ such that $q_v \equiv 1 \text{ mod }l$ and 
$\overline{r}|_{G_{F_\wv}}$ is trivial;
\item ``$\chi_v$-ramified deformations'' give rise to a problem 
$\cD_v^{\chi_v}$ for each place $v \in S$ such that $q_v \equiv 1 \text{ mod 
}l$ and $\overline{r}|_{G_{F_\wv}}$ is trivial, given the additional data of a 
tuple  $\chi_v = (\chi_{v, 1}, \dots, \chi_{v, n})$ of characters $\chi_{v, i} 
: k(v)^\times(l) \to k^\times$.
\item ``Unrestricted deformations'' give rise to a problem $\cD_v^\square$ for any $v \in S$. 
\end{itemize}
If $\cS$ is a global deformation problem, then we can define (as in \cite{Tho16}) a functor $\Def_\cS : \cC_\Lambda \to \text{Sets}$ of ``deformations of type $\cS$''. By definition, if $A \in \cC_\Lambda$, then $\Def_\cS(A)$ is the set of $\GL_n(A)$-conjugacy classes of homomorphisms $r : G_{F^+, S} \to \cG_n(A)$ lifting $\overline{r}$ such that $\nu \circ r = \chi$ and for each $v \in S$, $r|_{G_{F_\wv}} \in \cD_v(A)$. If $\overline{r}$ is \emph{Schur} (see \cite[Definition 3.2]{Tho16}), then the functor $\Def_\cS$ is represented by an object $R_\cS^\text{univ} \in \cC_\Lambda$.

\subsection{An erratum to \cite{Tho16}}\label{subsec_erratum}

We point out an error in \cite{Tho16}. We thank Lue Pan for bringing this to our attention. In \cite[Proposition 3.15]{Tho16}, it is asserted that the ring $R_v^1$ (representing the deformation problem $\cD_v^1$ for $v \in R$, defined under the assumptions $q_v \equiv 1 \text{ mod }l$ and $\overline{r}|_{G_{F_\wv}}$ trivial) has the property that $R_v^1 / (\lambda)$ is generically reduced. This is false, even in the case $n = 2$, as can be seen from the statement of \cite[Proposition 5.8]{Sho16} (and noting the identification $R_v^1 / (\lambda) = R_v^{\chi_v} / (\lambda)$). We offer the following corrected statement.
\begin{proposition}
	Let $\overline{R}_v^1$ denote the nilreduction of $R_v^1$. Then $\overline{R}_v^1 / (\lambda)$ is generically reduced.
\end{proposition}
\begin{proof}
	Let $\cM$ denote the scheme over $\cO$ of pairs of $n \times n$ matrices $(\Phi, \Sigma)$, where $\Phi$ is invertible, the characteristic polynomial of $\Sigma$ equals $(X - 1)^n$, and we have $\Phi \Sigma \Phi^{-1} = \Sigma^{q_v}$. Then $R_v^1$ can be identified with the completed local ring of $\cM$ at the point $(1_n, 1_n) \in \cM(k)$. By \cite[Theorem 23.9]{Mat07} (and since $\cM$ is excellent), it's enough to show that if $\overline{\cM}$ denotes the nilreduction of $\cM$, then $\overline{\cM} \otimes_\cO k$ is generically reduced.
	
	Let $\cM_1, \dots, \cM_r$ denote the irreducible components of $\cM$ with their reduced subscheme structure. According to \cite[Lemma 3.15]{Tho11}, each $\cM_i \otimes_\cO K$ is non-empty of dimension $n^2$, while the $\cM_i \otimes_\cO k$ are the pairwise distinct irreducible components of $\cM \otimes_\cO k$, and are all generically reduced. Let $\overline{\eta}_i$ denote the generic point of $\cM_i \otimes_\cO k$. Then $\overline{\eta}_i$ admits an open neighbourhood in $\cM$ not meeting any $\cM_j$ ($j \neq i$). Consequently we have an equality of local rings  $\cO_{\overline{\cM}, \overline{\eta}_i} = \cO_{\cM_i, \overline{\eta}_i}$, showing that $\cO_{\overline{\cM}, \overline{\eta}_i} / (\lambda)$ is reduced (in fact, a field). This shows that $\overline{\cM} \otimes_\cO k$ is generically reduced.
\end{proof}
	We now need to explain why this error does not affect the proofs of the two results in \cite{Tho16} which rely on the assertion that $R_v^1 / (\lambda)$ is generically reduced. The first of these is \cite[Proposition 3.17]{Tho16}, which states that the Steinberg deformation ring $R_v^{St}$ has the property that $R_v^{St} / (\lambda)$ is generically reduced. The proof of this result is still valid if one replaces $R_v^1$ there with $\overline{R}_v^1$. Indeed, we need only note that $R_v^{St}$ is $\cO$-flat (by definition) and reduced (since $R_v^{St}[1/l]$ is regular, by \cite[Lemma 3.3]{Tay08}). The map $R_v^1 \to R_v^{St}$ therefore factors through a surjection $\overline{R}_v^1 \to R_v^{St}$.
	
	The next result is \cite[Lemma 3.40, (2)]{Tho16}, which describes the irreducible components of the localization and completion of a ring $R^\infty$ at a prime ideal $P_\infty$. The ring $R^\infty$ has $R_v^1$ as a (completed) tensor factor, and the generic reducedness is used to justify an appeal to \cite[Proposition 1.6]{Tho16}. Since passing to nilreduction does not change the underlying topological space, one can argue instead with the quotient of $R^\infty$ where $R_v^1$ is replaced by $\overline{R}_v^1$. The statement of \cite[Lemma 3.40]{Tho16} is therefore still valid. 

\subsection{Pseudodeformations}\label{subsec_pseudodeformations}

In this section, we fix a global deformation problem
\[ \cS = (  F / F^+, S, \widetilde{S},  \Lambda, \overline{r}, \chi, \{ \cD_v \}_{v \in S} ) \]
such that $\overline{r}$ is Schur. We write $P_\cS \subset R_\cS^{\text{univ}}$ 
for the $\Lambda$-subalgebra topologically generated by the coefficients of 
characteristic polynomials of Frobenius elements $\Frob_w \in G_{F, S}$ ($w$ 
prime to $S$). The subring $P_\cS$ is studied in \cite[\S 3.4]{Tho16}, where it 
is shown using results of Chenevier that $P_\cS$ is a complete Noetherian local 
$\Lambda$-algebra and that the inclusion $P_\cS \subset R_{\cS}^\text{univ}$ is 
a finite ring map (see \cite[Proposition 3.29]{Tho16}).

In fact, more is true, as we now describe. Let $\overline{B} \in \GL_n(k)$ be 
the matrix defined by the formula $\overline{r}(c) = (\overline{B}, - 
\chi(c))\jmath \in \cG_n(k)$. Let $\overline{\rho} = \overline{r}|_{G_{F, S}}$, 
and suppose that there is a decomposition  $\overline{r}=\oplus_{i=1}^{\num} 
\overline{r}_i$ with $\overline{\rho}_i=\overline{r}_i|_{G_{F,S}}$ absolutely 
irreducible for each $i$. The representations $\overline{\rho}_i$ are pairwise 
non-isomorphic, because $\overline{r}$ is Schur (see \cite[Lemma~3.3]{Tho16}). 
We recall (\cite[Lemma~3.1]{Tho16}) that to give a lifting 
$r:G_{F^+,S}\rightarrow \cG_n(R)$ of $\overline{r}$ with $\nu\circ r = \chi$ is 
equivalent to giving the following data:
\begin{itemize}
\item A representation $\rho: G_{F,S}\rightarrow \GL_n(R)$ lifting $\overline{\rho} = \overline{r}|_{G_{F,S}}$. 
\item A matrix $B \in \GL_n(R)$ lifting $\overline{B}$ with $\prescript{t}{}{B} 
= 
-\chi(c)B$ and $\chi(\delta)B = \rho(\delta^c)B\prescript{t}{}{\rho}(\delta)$ 
for all $\delta \in G_{F,S}$.
\end{itemize}
The equivalence is given by letting $\rho = r|_{G_{F,S}}$ and $r(c) = 
(B,-\chi(c))\jmath$. Conjugating $r$ by $M \in \GL_n(R)$ takes $B$ to 
$MB\prescript{t}{}{M}$. Note that the matrix $B$ defines an isomorphism 
$\chi\otimes\rho^\vee \overset{\sim}{\rightarrow} \rho^c$. 

We embed the group $\mu_2^{\num}$ in $\GL_n(\cO)$ as block diagonal matrices, 
the $i^\text{th}$ block being of size $\dim_k \overline{\rho}_i$. We assume 
that the global deformation problem $\cS$ has the property that each local 
deformation problem $\cD_v \subset \cD_v^\square$ is invariant under 
conjugation by $\mu_2^{\num}$; this is the case for each of the local 
deformation problems recalled above. With this assumption, the group 
$\mu_2^{\num}$ acts on the ring $R_{\cS}^\text{univ}$ by conjugation of the 
universal deformation, and we have the following result.
\begin{proposition}\label{prop:pseudoinvariants}
\begin{enumerate} \item We have an equality $P_\cS = 
(R_\cS^\text{univ})^{\mu_2^{\num}}$.
\item Let $\frp \subset R_\cS^\text{univ}$ be a prime ideal, and let $\frq = 
\frp \cap P_\cS$. Let $E = \Frac R_\cS^\text{univ} / \frp$, and suppose that 
the associated representation $\rho_\frp = r_\frp|_{G_{F, S}}\otimes_A E : 
G_{F, S} \to \GL_n(E)$ is absolutely irreducible. Then $P_\cS \to 
R_\cS^\text{univ}$ is \'etale at $\frq$ and $\mu_2^{\num}$ acts transitively on 
the 
set of primes of $R_{\cS}^\text{univ}$ above $\frq$. 
\end{enumerate}
\end{proposition}
We first establish a preliminary lemma, before proving the proposition.
\begin{lemma}\label{lem:blocklemma}
Let $R = R_{\cS}^\text{univ}/(\ffrm_{P_{\cS}})$, and let $r: 
G_{F^+,S}\rightarrow \cG_n(R)$ be a representative of the specialisation of the 
universal deformation. Then, after possibly conjugating by an element of 
$1+M_n(\ffrm_R)$, $r|_{G_{F,S}}$ has (block) diagonal entries given by 
$\overline{\rho}_1,\ldots,\overline{\rho}_{\num}$, and the matrix $B$ defined 
above is equal to $\overline{B}$. (Note we are not asserting that the 
off-diagonal blocks of $r|_{G_{F, S}}$ are zero.)
\end{lemma}
\begin{proof}
We let $\overline{e}_1,\overline{e}_2,\ldots,\overline{e}_{\num} \in M_n(k)$ 
denote 
the standard idempotents decomposing $\overline{r}|_{G_{F,S}}$ into the block 
diagonal pieces $\overline{\rho}_1,\ldots,\overline{\rho}_{\num}$. We let $\cA 
\subset M_n(R)$ denote the image of $R[G_{F,S}]$ under $r$. The idempotents 
$\overline{e}_i$ lift to orthogonal idempotents $e_i$ in $\cA$ with $e_1 + 
\cdots + e_{\num} = 1$ and, after conjugating by an element of $1+M_n(\ffrm_R)$ 
we 
can assume that the $e_i$ are again the standard idempotents on $R^n$. 
Moreover, applying the first case of the proof of \cite[Lemma~1.8.2]{BCAst}, we 
can (and do) choose the $e_i$ so that they are fixed by the anti-involution 
$\star : \cA \to \cA$ given by the formula $M \mapsto B (\prescript{t}{}{M}) 
B^{-1}$. This implies that the matrix $B$ is block diagonal. We have $e_i\cA 
e_i = M_{n_i}(R)$ (see \cite[Lemma~1.4.3]{BCAst} and \cite[Theorem~2.22]{Che11}) 
and for each $i \ne j$ we have $e_i\cA e_j = M_{n_i,n_j}(\cA_{i,j})$, where 
$\cA_{i,j} \subset R$ is an ideal \cite[Proposition~1.3.8]{BCAst}. 

Since $\det \circ\, r = \det \circ\, \overline{r}$, 
Proposition \ref{prop:reducibleideal} shows 
that $\sum_{i \ne j}\cA_{i,j}\cA_{j,i} = 0$. This implies that for each $i$ the 
map \[R[G_{F,S}] \to M_{n_i}(R)\] given by \[ x \mapsto e_i r(x) e_i\] is an 
algebra homomorphism, and we get an $R$-valued lift of $\overline{\rho}_i$. By 
the uniqueness assertion in Proposition~\ref{prop:reducibleideal}~the determinant 
of this lift is equal to $\det \circ \overline{\rho}_i$. Since 
$\overline{\rho}_i$ is 
absolutely irreducible, it follows from \cite[Theorem~2.22]{Che11} that, after 
conjugating by a block diagonal matrix in $1+M_n(\ffrm_R)$, we can assume that 
the map \[ x \mapsto e_i r(x) e_i\] is induced by $\overline{\rho}_i$, which is 
the desired statement.

Finally, we consider the matrix $B$. We have already shown that $B$ is block 
diagonal. For $1\le i \le {\num}$ we denote the corresponding block of $B$ by 
$B_i$. It lifts a block $\overline{B}_i$ of $\overline{B}$. By Schur's lemma, 
we have $B_i = \beta_i\overline{B}_i$ for some scalars $\beta_i \in 1+\ffrm_R$. 
Since $2$ is invertible in $R$ we can find $\lambda_i \in 1+\ffrm_R$ with 
$\lambda_i^2 = \beta_i^{-1}$. Conjugating $r$ by the diagonal matrix with 
$\lambda_i$ in the $i$th block puts $r$ into the desired form.
\end{proof}
\begin{proof}[Proof of Proposition \ref{prop:pseudoinvariants}]
We begin by proving the first part. We again let $R = 
R_{\cS}^\text{univ}/(\ffrm_{P_{\cS}})$.  By Nakayama's lemma, it suffices to 
show that $R^{\mu_2^{\num}} = k$. Indeed, the natural map 
$(R_{\cS}^\text{univ})^{\mu_2^{\num}}/\ffrm_{P_{\cS}}(R_{\cS}^\text{univ})^{\mu_2^{\num}}
 \to R^{\mu_2^{\num}}$ 
is injective (i.e.~$(\ffrm_{P_{\cS}}R_{\cS}^\text{univ})^{\mu_2^{\num}} = 
 \ffrm_{P_{\cS}}(R_{\cS}^\text{univ})^{\mu_2^{\num}}$), since if $\sum_i m_i 
 x_i$ is ${\mu_2^{\num}}$-invariant, with $m_i \in \ffrm_{P_{\cS}}$ and $x_i 
 \in R_{\cS}^\text{univ}$, we have $\sum_i m_i x_i = \frac{1}{2^\num}\sum_i 
 m_i\sum_{\sigma \in \mu_2^\num} \sigma x_i$ which is an element of 
 $\ffrm_{P_{\cS}}(R_{\cS}^\text{univ})^{\mu_2^{\num}}$. Let $r : 
 G_{F^+, S} \to \cG_n(R)$ be a 
representative of the specialisation of the universal deformation satisfying 
the conclusion of Lemma \ref{lem:blocklemma}. Then $R$ is a finite $k$-algebra 
and is generated 
as a $k$-algebra by the matrix entries of $r$, hence the matrix entries of 
$\rho = r|_{G_{F, S}}$ (because $B = \overline{B}$). We recall the ideals 
$\cA_{i,j}\subset R$ appearing in the proof of Lemma \ref{lem:blocklemma}, 
which are generated by the block $(i,j)$ matrix entries of $\rho$. 
The conjugate self duality of $\rho$ is given by 
${}^t \rho(\delta) = \chi(\delta)\overline{B}^{-1} \rho((\delta^c)^{-1}) \overline{B}$, $\delta \in G_{F,S}$. 
Since $\overline{B}$ is block diagonal, we deduce that $\cA_{i,j}= \cA_{j,i}$. 
Since $\sum_{i \ne j}\cA_{i,j}\cA_{j,i} = 0$ we see that $\cA_{i,j}^2 = 
0$ for $i \ne j$. We deduce that $R$ is generated a $k$-module by $1 \in R$ 
	and products of 
the form \[a_{\cP} = \prod_{(i,j) \in \cP} a_{i,j}\] where $\emptyset \ne \cP 
\subset 
\{(i,j): 1\le i < 
j \le {\num}\} $ and $a_{i,j} \in \cA_{i,j}$ has action of $\mu_2^{\num}$ given 
by 
$((-1)^{\alpha_1},\ldots,(-1)^{\alpha_{\num}})a_{i,j} = 
(-1)^{\alpha_i+\alpha_j}a_{i,j}$. Suppose the action of $\mu_2^{\num}$ on 
$a_{\cP}$ 
is trivial. Then for each $1 \le i \le {\num}$, $i$ appears in an even number 
of 
elements of $\cP$. A product $a'_{j_1,j_2} = a_{1,j_1}a_{1,j_2}$ lies in 
$\cA_{j_1,j_2}$ and 
the action of $\mu_2^{\num}$ is given by 
$((-1)^{\alpha_1},\ldots,(-1)^{\alpha_{\num}})a'_{j_1,j_2} = 
(-1)^{\alpha_{j_1}+\alpha_{j_2}}a'_{j_1,j_2}$. Since $1$ appears in an 
even number of elements of $\cP$ we can `pair off' these elements and rewrite 
$a_\cP$ as a product \[a_{\cP'} 
= \prod_{(i,j) \in \cP'} a'_{i,j}\] where $\cP' \subset \{(i,j): 2\le i < 
	j \le {\num}\}$ and the action of $\mu_2^{\num}$ on $a'_{i,j}$ is given by 
	the same 
	formula as for $a_{i,j}$. 
Continuing in this manner, we deduce that $a_{\cP}$ is the product of an even 
number of elements of $\cA_{{\num}-1, {\num}}$, and thus equals $0$ since 
$\cA_{{\num}-1, {\num}}^2 = 0$.

The invariant subring $R^{\mu_2^{\num}}$ is equal to the $k$-submodule of $R$ 
generated by $\sum_{\sigma\in \mu_2^{\num}} \sigma x$ where $x$ runs over a set 
of 
$k$-module generators of $R$ (since $2$ is invertible in $k$). It follows from 
the above calculation that $R^{\mu_2^{\num}} = k$.

We now prove the second part. The diagonally embedded subgroup $\mu_2 \subseteq 
\mu_2^{\num}$ 
acts trivially on $R_{\cS}^\text{univ}$, 
so we have an induced action of $\mu_2^{\num}/\mu_2$. The first part together 
with \cite[\href{https://stacks.math.columbia.edu/tag/0BRI}{Tag 
0BRI}]{stacks-project} implies that $\mu_2^{\num} / \mu_2$ acts transitively on 
the set of primes of $R_{\cS}^\text{univ}$ above $\frq$. Let $R = 
R_{\cS}^\text{univ}/\frp$, and let $r_{\frp} \colon G_{F^+,S} \rightarrow 
\cG_n(R)$ be a representative of the specialisation of the universal 
deformation. By  
\cite[\href{http://stacks.math.columbia.edu/tag/0BST}{Tag~0BST}]{stacks-project},
 to finish the proof if will be enough to show that if $\sigma \in 
\mu_2^{\num}$, $\sigma(\frp) = \frp$, and $\sigma$ acts as the identity on $R$, 
then $\sigma \in \mu_2$.

If $\sigma \in \mu_2^{\num}$ corresponds to the block diagonal matrix $g\in 
\GL_n(\cO)$, then these conditions imply that $r_{\frp}$ and $g r_{\frp} 
g^{-1}$ are conjugate by an element $\gamma \in 1+M_n(\ffrm_R)$.
Since $r_{\frp}|_{G_{F,S}}\otimes E = \rho_{\frp}$ is absolutely irreducible, 
this implies that $g\gamma^{-1}$ is scalar, and so $g$ must also be scalar as 
$l>2$, hence $g \in \mu_2$. This completes the proof.
\end{proof}
For each partition $\{1, \dots, {\num}\} = \cP_1 \sqcup \cP_2$ with $\cP_1, 
\cP_2$ both non-empty, Proposition \ref{prop:reducibleideal} gives an ideal 
$I_{(\cP_1, \cP_2)} \subset P_\cS$ cutting out the locus where the determinant 
$\det r|_{G_{F, S}}$ is $(\cP_1, \cP_2)$-reducible. We write $I_\cS^{red} = 
\prod_{(\cP_1, \cP_2)} I_{(\cP_1, \cP_2)}$, an ideal of $P_\cS$.
\begin{lemma}\label{lem_reducibility_ideal}
	Let $\frp \subset R_\cS^{\text{univ}}$ be a prime ideal, and let $\frq = 
	\frp \cap P_{\cS}$. Let $A = R_{\cS}^{\text{univ}} / \frp$, 
	$E = \Frac A$. Then $\rho_\frp = r_\frp|_{G_{F, S}}\otimes_A E$ is 
	absolutely 
	irreducible if and 
	only if $I_\cS^{red} \not\subset \frq$.
\end{lemma}
\begin{proof}
	If $I_\cS^{red} \subset \frq$, then $I_{(\cP_1, \cP_2)} \subset \frq$ for 
	some proper partition $(\cP_1, \cP_2)$. Then Proposition 
	\ref{prop:reducibleideal} implies that $\det r_\frp$ admits a decomposition 
	$\det \circ r_\frp|_{G_{F,S}} = D_1 D_2$ for two determinants $D_i : 
	A[G_{F, S}] 
	\to 
	M_{n_i}(A)$. Then \cite[Corollary 2.13]{Che11} implies that $\rho_\frp$ is 
	not absolutely irreducible.
	
	Suppose conversely that $\rho_\frp$ is not absolutely irreducible. 
	Let $J_{(\cP_1, \cP_2)}$ denote the image of $I_{(\cP_1, \cP_2)}$ in $A$. 
	We must show that some $J_{(\cP_1, \cP_2)}$ is zero. Let $\cA$ denote 
	the image of $A[G_{F, S}]$ in $M_n(A)$ under $r_\frp|_{G_{F, S}}$. 
	According to \cite[Theorem 1.4.4]{BCAst}, we can assume that $\cA$ has the 
	form
		\begin{equation}\label{eqn_image_algebra}
	\cA = \begin{pmatrix}
	M_{n_1}(A) & M_{n_1,n_2}(\mathcal{A}_{1,2}) & \cdots & 
	M_{n_1,n_{\num}}(\mathcal{A}_{1,{\num}})\\
	M_{n_2,n_1}(\mathcal{A}_{2,1}) & M_{n_2}(A) & \cdots & 
	M_{n_2,n_{\num}}(\mathcal{A}_{2,{\num}})\\
	\vdots & \vdots & \ddots & \vdots \\
	M_{n_{\num},n_1}(\mathcal{A}_{{\num},1}) & 
	M_{n_{\num},n_2}(\mathcal{A}_{{\num},2}) & \cdots & M_{n_{\num}}(A)
	\end{pmatrix},
	\end{equation}
	where each $\mathcal{A}_{i, j}$ is a fractional ideal of $E$. Consequently $\cA \otimes_A E$ has the form
		\begin{equation}
	\cA \otimes_A E= \begin{pmatrix}
	M_{n_1}(E) & M_{n_1,n_2}(\mathcal{E}_{1,2}) & \cdots & 
	M_{n_1,n_{\num}}(\mathcal{E}_{1,{\num}})\\
	M_{n_2,n_1}(\mathcal{E}_{2,1}) & M_{n_2}(E) & \cdots & 
	M_{n_2,n_{\num}}(\mathcal{E}_{2,{\num}})\\
	\vdots & \vdots & \ddots & \vdots \\
	M_{n_{\num},n_1}(\mathcal{E}_{{\num},1}) & 
	M_{n_{\num},n_2}(\mathcal{E}_{{\num},2}) & \cdots & M_{n_{\num}}(E)
	\end{pmatrix},
	\end{equation}
	where each $\mathcal{E}_{i, j} = \mathcal{A}_{i, j} \otimes_A E$ equals 
	either $E$ or $0$. Let $f_i \in M_n(E)$ denote the matrix with 1 in the 
	$(i, i)$ entry and 0 everywhere else. If $\rho_\frp$ is not absolutely 
	irreducible 
	then $\cA \otimes_A E$ is a proper subspace of $M_n(E)$, so there exists 
	$i$ such that $(\cA \otimes_A E) f_i \subset M_n(E) f_i$ is a proper 
	subspace. Since $M_n(E) f_i$ is isomorphic as $M_n(E)$-module to the 
	tautological representation $E^n$, this implies that the $\cA \otimes_A 
	E$-module $E^n$ admits a proper invariant subspace. After permuting the 
	diagonal blocks, we can assume 
	that this subspace is $E^{n_1 + \dots + n_s}$ for some $s < {\num}$ 
	(included as 
	the subspace of $E^n$ where the last $n_{s+1} + \dots + n_{\num}$ entries 
	are 
	zero). Otherwise said, the spaces $\mathcal{E}_{i, j}$ for $i > s$, $j \leq 
	s$ are zero. If $\cP_1 = \{ 1, \dots, s \}$ and $\cP_2 = \{ s+1, \dots, 
	\num\}$ then this implies $\cJ_{(\cP_1, \cP_2)} \otimes_A E = 0$ and hence 
	(as 
	$A$ is a domain) $\cJ_{(\cP_1, \cP_2)} = 0$. This completes the proof. 
\end{proof}
\begin{lemma}\label{lem_reducible_implies_reducible}
	Let $\frp \subset R_\cS^{\text{univ}}$ be a prime ideal, $A = 
	R_\cS^{\text{univ}} / \frp$, $E = \Frac A$. Then $r_\frp\otimes_A E$ is 
	Schur and if $r_\frp|_{G_{F, S}} \otimes_A E$ is not absolutely 
	irreducible, then $r_\frp$ is equivalent (i.e.~conjugate by an element in 
	$1+ M_n(\ffrm_A)$) to a type-$\cS$ 
	lifting 
	of the form $r_\frp = r_1 \oplus r_2$, where $r_i : G_{F^+, S} \to 
	\cG_{m_i}(A)$ and $m_1 m_2 \neq 0$.
\end{lemma}
\begin{proof}
	We argue, as in the proof of Lemma \ref{lem_reducibility_ideal}, using the 
	image $\cA \subset M_n(A)$ of $A[G_{F, S}]$, which is a generalized matrix 
	algebra. Suppose given $G_{F, S}$-invariant subspaces $E^n \supset W_1 
	\supset W_2$ such that $W_2$ and $E^n / W_1$ are irreducible. We can assume 
	that $\cA$ has the form (\ref{eqn_image_algebra}) and that this 
	decomposition is block upper triangular (perhaps with respect to a coarser 
	partition than $n = n_1 + \dots + n_{\num}$), and moreover than the first 
	block corresponds to $W_2$, while the last block corresponds to $E^n / 
	W_1$. In particular, $W_2$ and $E^n / W_1$ are even absolutely irreducible. 
	Note that there can be no isomorphism $W_2^{c \vee}(\nu \circ r_\frp) \cong 
	E^n / W_1$; if there was, then it would imply an identity of $A$-valued 
	determinants, which we could reduce modulo $\ffrm_A$ to obtain an identity 
	$\{ \rho_{i} \} = \{ \rho_j \}$ of sets of irreducible constituents of 
	$\overline{r}|_{G_{F, S}}$. Since these appear with multiplicity one, this 
	is impossible. This all shows that $r_\frp \otimes_A E$ is necessarily 
	Schur.
	
	Now suppose that $r_\frp|_{G_{F, S}}\otimes_A E$ is not absolutely 
	irreducible. After permuting the diagonal blocks of $\overline{r}$, we can 
	assume that there is some $1 \leq m \leq \num$ such that $\cA_{i, j} = 0$ 
	for 
	$i > m$, $j \leq m$. The existence of the conjugate self-duality of 
	$r_\frp$ implies (cf. \cite[Lemma 1.8.5]{BCAst}) that $\cA_{j, i} = 0$ in 
	the same range, giving a decomposition $r_\frp|_{G_{F, S}} = \rho_1 \oplus 
	\rho_2$ of representations over $A$. Since $r_\frp \otimes_A E$ is Schur 
	the conjugate self-duality of $r_\frp$ must make $\rho_1$ and $\rho_2$ 
	orthogonal, 
	showing that $r_\frp$ itself decomposes as $r_\frp = r_1 \oplus r_2$.
\end{proof}
\subsection{Dimension bounds}\label{subsec_dimension_bounds}

We now suppose that $S$ admits a decomposition $S = S_l \sqcup S(B) \sqcup R \sqcup S_a$, where:
\begin{itemize}
\item For each $v \in S(B) \cup R$, $q_v \equiv 1 \text{ mod }l$ and 
$\overline{r}|_{G_{F_\wv}}$ is trivial. 
\item For each $v \in S_a$, $q_v \not\equiv 1 \text{ mod }l$, $\overline{r}|_{G_{F_\wv}}$ is unramified, and $\overline{r}|_{G_{F_\wv}}$ is scalar. (Then any lifting of $\overline{r}|_{G_{F_\wv}}$ is unramified.)
\end{itemize}
We consider the global deformation problem
\[ \cS = (  F / F^+, S, \widetilde{S},  \Lambda, \overline{r}, \chi, \{ \cD_v^\triangle \}_{v \in S_l} \cup \{ \cD_v^\text{St} \}_{v \in S(B)} \cup \{ \cD_v^1 \}_{v \in R} \cup \{ \cD_v^\square \}_{v \in S_a} ), \]
where $\overline{r}$ is assumed to be Schur. We define quantities $d_{F, 0} = d_0$ and $d_{F, l} = d_l$ as follows. Let $\Delta$ denote the Galois group of the maximal abelian pro-$l$ extension of $F$ which is unramified outside $l$, and let $\Delta_0$ denote the Galois group of the maximal abelian pro-$l$ extension of $F$ which is unramified outside $l$ and in which each place of $S(B)$ splits completely. We set
\[ d_0 = \dim_{\bbQ_l} \ker( \Delta[1/l] \to \Delta_0[1/l] )^{c = -1} \]
and
\[ d_l = \inf_{v \in S_l} [F^+_v : \bbQ_l]. \]
\begin{lemma}\label{lem_reducible_dimension_bound}
Suppose that $d_l > n(n-1)/2 + 1$. Let $A \in \cC_\Lambda$ be a finite 
$\Lambda$-algebra and let $r : G_{F^+, S} \to \cG_n(A)$ be a lifting of 
$\overline{r}$ of type $\cS$. Then $\dim A / (I_\cS^{red}, \lambda) \leq n[F^+ 
: \bbQ] - d_0$.
\end{lemma}
\begin{proof}
We can assume without loss of generality that $A = A / (I_\cS^{red}, \lambda)$, and must show that $\dim A \leq [F^+ : \bbQ] - d_0$. Since $A$ is Noetherian and we are interested only in dimension, we can assume moreover that $A$ is integral. Let $E = \Frac(A)$. Then (Lemma \ref{lem_reducible_implies_reducible}) we can find a non-trivial partition $n = n_1 + n_2$ and homomorphisms $r_i : G_{F^+, S} \to \cG_{n_i}(A)$ ($i = 1, 2$) such that $r = r_1 \oplus r_2$. 

Let $\overline{E}$ be a choice of algebraic closure of $E$. Our condition on $d_l$ means that we can appeal to \cite[Corollary 3.12]{Tho16} (characterization of $A$-valued points of $\cD_v^\triangle$ for each $v \in S_l$). This result implies the existence for each $v \in S_l$ of an increasing filtration
\[ 0 \subset \Fil^1_v \subset \Fil^2_v \subset \dots \subset \Fil^n_v = \overline{E}^n \]
of $r|_{G_{F_\wv}} \otimes_A \overline{E}$ by $G_{F_\wv}$-invariant subspaces, such that each $\Fil^i_v / \Fil^{i-1}_v$ is 1-dimensional, and the character of $I^\text{ab}_{F_\wv}(l)$ acting on this space is given by composing the universal character $\psi_v^i : I^\text{ab}_{F_\wv}(l) \to \Lambda_v^\times$ with the homomorphism
\[ \Lambda_v \to \Lambda \to A \to \overline{E}. \]
The direct sum decomposition of $r$ leads to a decomposition $r|_{G_{F_\wv}} = r_1|_{G_{F_\wv}} \oplus r_2|_{G_{F_\wv}}$. Define $F_v^i = \Fil_v^i \cap r_1|_{G_{F_\wv}} \otimes_A \overline{E}$ and $G_v^i = \Fil_v^i \cap r_2|_{G_{F_\wv}} \otimes_A \overline{E}$. Then $F_v^\bullet$ and $G_v^\bullet$ are increasing filtrations of $\overline{E}^{n_1}$ and $\overline{E}^{n_2}$, respectively, with graded pieces of dimension at most 1. We write $\sigma_v$ for the bijection
\[ \sigma_v : \{ 1, \dots, n_1 \} \sqcup \{ 1, \dots, n_2 \} \to \{1, \dots, n \} \]
which is increasing on $\{ 1, \dots, n_1 \}$ and $\{ 1, \dots, n_2 \}$ and which has the property that $\sigma_v( \{ 1, \dots, n_1 \} )$ is the set of $i \in \{1, \dots, n\}$ such that the graded piece $F_v^i / F_v^{i-1}$ is non-trivial. 

Let $\Lambda_{v, 1}$, $\Lambda_{v, 2}$ denote the analogues of the algebra $\Lambda_v$ in dimensions $n_1$ and $n_2$, respectively. The bijection $\sigma_v$ determines in an obvious way an isomorphism $\Lambda_{v, 1} \widehat{\otimes} \Lambda_{v, 2} \cong \Lambda_v$. Applying again \cite[Corollary 3.12]{Tho16}, we see that with this structure on $A$ of $\Lambda_{v, i}$-algebra, each homomorphism $r_i|_{G_{F_\wv}} : G_{F_\wv} \to \GL_{n_i}(A)$ is of type $\cD_v^\triangle(A)$. Similarly if we define $\Lambda_i = \widehat{\otimes}_{v \in S_l} \Lambda_{v, i}$ then the collection of bijections $(\sigma_v)_{v \in S_l}$ determines an isomorphism $\Lambda_1 \widehat{\otimes}_\cO \Lambda_2 \cong \Lambda$. 

We also define $\Lambda_{v, 0} = \cO \llbracket I^\text{ab}_{F_\wv}(l) \rrbracket$ and $\Lambda_0 = \widehat{\otimes}_{v \in S_l} \Lambda_{v, 0}$.
Then there are natural maps $\Lambda_0 \to \Lambda_i$ classifying the characters $\prod_{j=1}^{n_i} \psi_v^{j} : I^\text{ab}_{F_\wv}(l) \to \Lambda_{v, i}^\times$. Let $\overline{\chi}_i = \det \overline{r}_i|_{G_{F, S}}$.  We get a commutative diagram
\[ \xymatrix{ \Lambda_{0} \widehat{\otimes}_\cO \Lambda_{0} \ar[r] \ar[d]  & k\llbracket \Delta  / (c+1) \rrbracket \widehat{\otimes}_k k\llbracket \Delta / (c+1) \rrbracket \ar[d] \\ \Lambda_{1} \widehat{\otimes}_\cO \Lambda_{2} \ar[r] & A}, \]
where the map $k\llbracket \Delta / (c+1) \rrbracket \widehat{\otimes}_k k\llbracket \Delta /  (c+1) \rrbracket \to A$ classifies the pair of $A$-valued characters $(\chi_1, \chi_2) = ( \overline{\chi}_1^{-1} \det r_1|_{G_F}, \overline{\chi}_2^{-1} \det r_2|_{G_F})$ of the group $\Delta / (c+1)$. The natural map  $\Lambda_{0}/(\lambda) \to k \llbracket  \Delta  / (c+1) \rrbracket$ is finite 
(and dominant). (Note that $\det r_1$ and $\det r_2$ are unramified at places of $S(B) \cup R \cup S_a$ because of our choices of deformation problem.)

We now use the existence of the places $S(B)$. For each place $v \in S(B)$, imposing the Steinberg condition on $r_1 \oplus r_2$ determines a relation $\chi_1(\Frob_\wv)^{n_2} = \chi_2(\Frob_\wv)^{n_1}$ in $A$. Let $\cE$ denote the quotient of the group $\Delta  / (c+1) \times \Delta / (c+1)$ by the $\bbZ_l$-submodule generated by the elements $(n_2 \Frob_\wv, -n_1 \Frob_\wv)$ ($v \in S(B)$). Then $\dim_{\bbQ_l} \cE[1/l] = \dim_{\bbQ_l} (\Delta / (c+1)\times \Delta / (c+1))[1/l] - d_0$, and we in fact have a commutative diagram
\[ \xymatrix{ \Lambda_{0} \widehat{\otimes}_\cO \Lambda_{0} \ar[r] \ar[d]  & k\llbracket \cE \rrbracket \ar[d] \\ \Lambda_{1} \widehat{\otimes}_\cO \Lambda_{2} \ar[r] & A}. \]
We deduce that the map $\Lambda \cong \Lambda_1 \widehat{\otimes}_\cO \Lambda_2 
\to A$ factors through the quotient $\Lambda_1 \widehat{\otimes}_\cO \Lambda_2 
\otimes_{\Lambda_{0} \widehat{\otimes}_\cO \Lambda_{0}} k\llbracket \cE 
\rrbracket$ of dimension $n[F^+ : \bbQ] - d_0$. Using finally that $A$ is a 
finite $\Lambda$-algebra, we see that $\dim A$ must satisfy the same estimate. 
This concludes the proof.
\end{proof}

\begin{definition}\label{def_generic_homomorphism} Let $A \in \cC_\Lambda$, and let $r : G_{F^+, S} \to \cG_n(A)$ be a homomorphism of type $\cS$. We say that $r$ is generic at $l$ if it satisfies the following two conditions:
\begin{itemize}
\item For each $v \in S_l$, the universal characters $\psi^v_{ 1}, \dots, \psi^v_{n}  : I^\text{ab}_{F_\wv}(l) \to A^\times$ are distinct.
\item There exists $v \in S_l$ and $\sigma \in I^\text{ab}_{F_\wv}(l)$ such 
that the elements $\psi^v_1(\sigma),\dots, \psi^v_n(\sigma) \in A^\times$ 
satisfy no non-trivial $\bbZ$-linear relation. 
\end{itemize}
We say that $r$ is generic if it is generic at $l$, $A$ is a domain, and $r|_{G_F} \otimes_A \Frac(A)$ is absolutely irreducible.
\end{definition}
\begin{lemma}\label{lem_generic_dimension_bound}
There exists a countable collection of ideals $I_i \subset \Lambda/(\lambda)$ 
($i = 1, 2, \dots$) with the following properties:
\begin{enumerate}
\item For each $i = 1, 2, \dots$, we have $\dim \Lambda / I_i \leq n[F^+ : 
\bbQ] - d_l$.
\item Suppose that $A \in \cC_\Lambda$ and $r : G_{F^+} \to \cG_n(A)$ is a lifting of type $\cS$ which is not generic at $l$. Then there exists $i \geq 1$ such that $I_i A = 0$.
\end{enumerate}
\end{lemma}
\begin{proof}
For each $v \in S_l$, let $d_v = [ F^+_v : \bbQ_l]$ and let $\sigma_{v, 1}, \dots, \sigma_{v, d_v} \in I^\text{ab}_{F_{\wv}}(l)$ be elements which project to a $\bbZ_l$-basis of the $l$-torsion-free quotient of this finitely generated $\bbZ_l$-module.

For each $1 \leq i < j \leq n$ and $v \in S_l$, we define an ideal $I(i, j, v) = (\lambda, \psi_v^i(\sigma_{v, k}) - \psi_v^j(\sigma_{v, k}))_{k = 1, \dots, d_v}$. Then $\dim \Lambda / I(i, j, v) = n [F^+ : \bbQ] - d_v$ and if $\frp \subset \Lambda$ is a prime of characteristic $l$ which does not contain $I(i, j, v)$, 
then the characters $\psi_v^i \text{ mod } \frp$ and $\psi_v^j \text{ mod } \frp$ are distinct. 

Suppose given for each $v \in S_l$ an $n \times d_v$ matrix $A_v$ with integer entries, and with at least one non-zero entry in each column. Then we define an ideal $J((A_v)_{v \in S_l}) \subset \Lambda$ as the ideal generated by $\lambda$ and all the elements
\[ \prod_{i=1}^n \psi_v^i(\sigma_{v, j})^{A_{v, i, j}} - 1 \text{ }(v \in S_l, j = 1, \dots, d_v). \]
Then $\dim \Lambda / J((A_v)_{v \in S_l}) = (n-1)[F^+ : \bbQ]$ and if $\frp \subset \Lambda$ is a prime of characteristic $l$ not containing any of the ideals $J((A_v)_{v \in S_l})$, then there exists $v \in S_l$ and $1 \leq j \leq d_v$ such that the elements $\psi_v^1(\sigma_{v, j}), \dots, \psi_v^n(\sigma_{v, j}) \in (\Lambda / \frp)^\times$ satisfy no non-trivial $\bbZ$-linear relation. 

The lemma is completed on taking our countable collection of ideals $I_i$ to consist of all the ideals $I(i, j, v)$ and $J((A_v)_{v \in S_l})$ defined above. 
\end{proof}
Combining the previous lemmas, we obtain the following result.
\begin{lemma}\label{lem_combined_dimension_bound}
Suppose that $d_l > n(n-1)/2 + 1$ and that $A \in \cC_\Lambda$ is a finite 
$\Lambda$-algebra such that $\dim A / (\lambda) > \sup (n[F^+ : \bbQ] - d_0, 
n[F^+ : \bbQ] - d_l)$. Let $r : G_{F^+, S} \to \cG_n(A)$ be a homomorphism of 
type $\cS$. Then we can find a prime ideal $\frp \subset A$ of dimension 1 and 
characteristic $l$ such that, writing $r_\frp = r \text{ mod } \frp: G_{F^+} 
\to \cG_n(A / \frp)$, $r_\frp$ is generic. 
\end{lemma}
\begin{proof}
Replacing $A$ by $A / (\lambda)$, we can assume that $A$ is a finite $\Lambda / (\lambda)$-algebra. Let $I_i$ ($i = 1, 2, \dots)$ be the countable collection of ideals of $\Lambda / (\lambda)$ defined in Lemma \ref{lem_generic_dimension_bound}. Then $\dim A / I_i A \leq n [ F^+ : \bbQ] - d_l$ for each $i = 1, 2, \dots$. Let $I_0 = I_\cS^{red} A$; then Lemma \ref{lem_reducible_dimension_bound} shows that $\dim A / I_0 \leq n [F^+ : \bbQ] - d_0$.

Applying \cite[Lemma 1.9]{Tho16}, we can find a prime ideal $\frp \subset A$ of dimension 1 (necessarily of characteristic $l$) such that $\frp$ does not contain any of the ideals $I_0$, $I_1 A$, $I_2 A, \dots$. By construction, the homomorphism $r_\frp$ is then generic. 
\end{proof}

\section{Automorphic forms and Hecke algebras on unitary groups}\label{sec_automorphic_forms}

\subsection{Hecke algebras}

We introduce automorphic forms on unitary groups and related Hecke algebras, using exactly the same notation as in \cite[\S 4]{Tho16}. This means we start with the following data:
\begin{itemize}
\item An integer $n \geq 1$, an odd prime $l$, and a totally imaginary CM 
number field $L$ with totally real subfield $L^+$. We write $S_l$ for the set 
of $l$-adic places of $L^+$. We assume that $L / L^+$ is everywhere unramified. 
(We note that this implies that $[L^+ : \bbQ]$ is even. Indeed, the 
quadratic character of 
$(L^+)^\times\backslash\bbA_{L^+}^\times/\widehat{\cO}_{L^+}^\times$ cutting 
out $L$ has non-trivial restriction to $(L^+_v)^\times$ for each $v|\infty$ but 
is 
trivial 
on $(-1)_{v|\infty} \in (L^+_\infty)^\times$.)
\item A coefficient field $K \subset \overline{\bbQ}_l$ which contains the image of all embeddings $L \hookrightarrow \overline{\bbQ}_l$. 
\item A finite set $S(B)$ of finite, prime-to-$l$ places of $L^+$ which all 
split in $L$. If $n$ is even then we assume that $n [ L^+ : \bbQ] / 2 + | S(B) 
|$ is also even. We allow the possibility that $S(B)$ may be empty. (Since 
$[L^+:\bbQ]$ is even, we 
are really just asking that if $n$ is even, then $|S(B)|$ is even.)
\end{itemize}
We can then find a central simple algebra $B$ over $L$ equipped with an involution $\dagger$, such that $\dim_L B = n^2$, $B^\text{op} \cong B \otimes_{L, c} L$, $B$ is split outside $S(B)$, $B$ is a division algebra locally at places $w | S(B)$ of $L$, and $\dagger|_L = c$. We can moreover assume that the unitary group $G$ over $L^+$ defined by the formula ($R$ an $L^+$-algebra)
\[ G(R) = \{ g \in (B \otimes_{L^+} R)^\times \mid g g^{\dagger \otimes 1} = 1 \} \]
has the property that $G(L^+ \otimes_\bbQ \bbR)$ is compact and for each finite place $v\not\in S(B)$ of $L^+$, $G_{L^+}$ is quasi-split (hence unramified).

We consider automorphic forms on the group $G$. To define our Hecke algebras, we need to fix the following additional choices:
\begin{itemize}
\item A finite set $R$ of finite places of $L^+$, disjoint from $S_l \cup S(B)$ and split in $L$, and such that for each $v \in R$, $q_v \equiv 1 \text{ mod }l$.
\item For each $v \in R$, we fix a choice of $n$-tuple of characters $\chi_{v, 1}, \dots, \chi_{v, n} : k(v)^\times \to \cO^\times$ which are trivial mod $\varpi$.
\item A finite set $T$ of finite places of $L^+$ split in $L$, containing $S_l \cup S(B) \cup R$. For each $v \in T$ we fix a choice of place $\wv$ of $L$ lying above $v$, and set $\widetilde{T} = \{ \wv \mid v \in T \}$.
\end{itemize}
It is also convenient to fix a choice of order $\cO_B \subset B$ such that $\cO_B^\dagger = \cO_B$ and for each place $v$ of $L^+$ which splits $v = w w^c$ in $L$, $\cO_{B_w}$ is a maximal order in $B_w$. We use this maximal order $\cO_B$ to extend $G$ to a group scheme over $\cO_{L^+}$. If $v \in T$, then a choice of isomorphism $\cO_{B_\wv} \cong M_n(\cO_{F_\wv})$ determines an isomorphism $\iota_\wv : G_{\cO_{F^+_v}} \to \Res_{\cO_{F_\wv} / \cO_{F_v^+}} \GL_n$, and we fix such a choice.  

Now suppose given an open compact subgroup $U = \prod_v U_v \subset G(\bbA_{L^+}^\infty)$ which satisfies the following conditions:
\begin{itemize}
\item There exists a finite place $v \not\in S_l$ of $L^+$ such that $U_v$ contains no non-trivial elements of finite order. (In other words, $U$ is sufficiently small in the terminology of \cite{Tho16}).
\item If $v \not\in T$ is a finite place of $L^+$ split in $L$, then $U_v = G(\cO_{L^+_v})$.
\item If $v \in R$, then $U_v = \iota_\wv^{-1} \Iw(\wv)$, where $\Iw(\wv) \subset \GL_n(\cO_{F_\wv})$ denotes the standard Iwahori subgroup.
\item If $v$ is a finite place of $L^+$ inert in $L$, then $U_v$ is a hyperspecial maximal compact subgroup of $G({L_v^+})$.
\end{itemize}
In this case we have defined in \cite[Definition 4.2]{Tho16} a Hecke algebra $\bbT^T_\chi(U(\frl^\infty), \cO)$. It is a finite faithful $\Lambda$-algebra, defined as an inverse limit of Hecke algebras which act on spaces of ordinary automorphic forms on $G$ with coefficients in $\cO$. According to \cite[Proposition 4.7]{Tho16}, to any maximal ideal $\ffrm \subset \bbT^T_\chi(U(\frl^\infty), \cO)$ of residue field $k$, one can associate a continuous semi-simple representation $\overline{\rho}_\ffrm : G_{F, T} \to \GL_n(k)$, uniquely characterized up to $\GL_n(k)$ conjugacy by a formula for the characteristic polynomials of Frobenius elements in terms of Hecke operators. 

\subsection{Deformation rings}\label{subsec_set-up_for_R_equals_T}
We connect the Hecke algebras defined in the previous section to deformation rings only under the following assumptions (1) -- (3):
\begin{enumerate}
\item $T$ has the form $T = S_l \sqcup S(B) \sqcup R \sqcup S_a$, where $S_a$ 
is a non-empty set of places of odd residue characteristic which are absolutely 
unramified and not split in $L(\zeta_l)$. For every $v \in S_l$, 
$[F_v^+:\bbQ_l] 
> n(n-1)/2 + 1$.
\item $U = \prod U_v \subset G(\bbA_{L^+}^\infty)$ is an open compact subgroup such that if $v \in S_a$, then $U_v = \iota_\wv^{-1}( \ker( \GL_n(\cO_{L_\wv}) \to \GL_n(k(\wv)) ) )$, and if $v \in S(B)$, then $U_v$ is the unique maximal compact subgroup of $G(F_v^+)$. Since $S_a$ is non-empty, this forces $U$ to be sufficiently small. 
\item $\ffrm \subset \bbT^T_\chi(U(\frl^\infty), \cO)$ is a maximal ideal of residue field $k$ such that $\overline{\rho}_\ffrm : G_{F, T} \to \GL_n(k)$ satisfies the following conditions:
\begin{itemize}
\item If $v \in S_a$, then $\overline{\rho}_\ffrm|_{G_{L_\wv}}$ is unramified and $\overline{\rho}_\ffrm(\Frob_\wv)$ is scalar.
\item If $v \in S_l \cup R \cup S(B)$, then $\overline{\rho}_\ffrm|_{G_{L_\wv}}$ is the trivial representation.
\item For each $v \in S(B) \cup R$, $q_v \equiv 1 \text{ mod }l$.
\item Let $\overline{\rho}_\ffrm = \oplus_{i=1}^\num \overline{\rho}_i$ denote 
the decomposition into simple constituents. Then each $\overline{\rho}_i$ is 
absolutely irreducible, satisfies $\overline{\rho}^c_i \cong 
\overline{\rho}_i^\vee \otimes \epsilon^{1-n}$, and if $i \neq j$ then 
$\overline{\rho}_i \not\cong \overline{\rho}_j$. (Equivalently, the maximal 
ideal $\ffrm$ is residually Schur in the sense of \cite[Definition 4.8]{Tho16}.)
\end{itemize}

\end{enumerate}
Then (\cite[Proposition 4.9]{Tho16}) $\overline{\rho}_\ffrm$ extends to a  homomorphism $\overline{r}_\ffrm : G_{L^+, T} \to \cG_n(k)$ such that $\nu \circ \overline{r}_\ffrm = \epsilon^{1-n} \delta_{L / L^+}^{n}$, and which is Schur. We can consider the global deformation problem
\[ \cS_\chi = (  L / L^+, T, \widetilde{T}, \Lambda, \overline{r}_\ffrm, \epsilon^{1-n} \delta_{L / L^+}^{n}, \{ R_v^\triangle \}_{v \in S_l} \cup \{ R_v^{\chi_v} \}_{v \in R} \cup \{ R_v^\text{St} \}_{v \in S(B) } \cup \{ R_v^\square \}_{v \in S_a} ). \]
Then there are $\Lambda$-algebra homomorphisms $P_{\cS_\chi} \to R_{\cS_\chi}$ and $P_{\cS_\chi} \to \bbT^T_\chi(U(\frl^\infty), \cO)_\ffrm$ (see \cite[Proposition 4.13]{Tho16}). We define $J_{\cS_\chi} = \ker( P_{\cS_\chi} \to \bbT^T_\chi(U(\frl^\infty), \cO)_\ffrm )$. 
\begin{theorem}\label{thm_application_of_Taylor-Wiles}
Let $\frp \subset R_{\cS_1}^\text{univ}$ be a prime ideal of dimension 1 and 
characteristic $l$. Let $A = R_{\cS_1}^\text{univ} / \frp$, and suppose that 
the following conditions are satisfied:
\begin{enumerate}
	\item $J_{\cS_1}R_{\cS_1}^\text{univ} \subset \frp$.
\item The representation $r_\frp$ is generic, in the sense of Definition \ref{def_generic_homomorphism}.
\item For each $v \in R$, $r_\frp|_{G_{L_\wv}}$ is the trivial representation and if $l^N || q_v - 1$ then $l^N > n$. For each $v  \in S(B)$, $r_\frp|_{G_{L_\wv}}$ is unramified and $r_\frp(\Frob_\wv)$ is a scalar matrix. 
\item $\overline{r}_\ffrm|_{G_{L, S}}$ is primitive. 
\item $\zeta_l \not\in L$, $\overline{r}_\ffrm|_{G_{L^+(\zeta_l)}}$ is Schur, 
and $\overline{r}_\ffrm(G_{L, S})$ does not have a quotient of order $l$.
\item $l > 3$ and $l \nmid n$.
\end{enumerate}
Let $Q \subset R_{\cS_1}^\text{univ}$ be a prime such that $Q \subset \frp$. 
Then $J_{\cS_1} R_{\cS_1}^\text{univ} \subset Q$. 
\end{theorem}
\begin{proof}
Let $\frq = \frp \cap P_{\cS_1}$. We can assume, after twisting by a character, 
that $\Frac R_{\cS_1}^\text{univ} / \frp = \Frac P_{\cS_1} / \frq$. (Apply 
\cite[Lemma 3.38]{Tho16} and \cite[Corollary 4.14]{Tho16}.)

In the case that $\overline{r}_\ffrm|_{G_{L, S}}$ has two irreducible 
constituents, the theorem follows on combining \cite[Corollary 5.7]{Tho16} 
(existence of Taylor--Wiles primes under a subset of the hypotheses listed 
here) and \cite[Corollary 4.20]{Tho16} (the assertion $J_{\cS_1} 
R_{\cS_1}^\text{univ} \subset Q$ assuming existence of Taylor--Wiles primes 
and, in addition, that $\overline{r}_\ffrm|_{G_{L, S}}$ has two irreducible 
constituents).

The proof of \cite[Corollary 4.20]{Tho16} can easily be modified to allow 
$\overline{r}_\ffrm|_{G_{L, S}}$ to have an arbitrary number $\num$ of 
irreducible constituents: one just need replace the $\mu^2_2$ action there by 
the $\mu_2^{\num}$ action on $R_{\cS_1}^\text{univ}$ described in \S 
\ref{subsec_pseudodeformations}, and replace the appeal to \cite[Proposition 
3.29]{Tho16} with an appeal to Proposition \ref{prop:pseudoinvariants} of this 
paper. We omit the details. 
\end{proof}
\section{Propogation of potential pro-automorphy}\label{sec_mainargument}

In this section we use Theorem \ref{thm_application_of_Taylor-Wiles} (informally, $R = \bbT$ locally at generic primes) to prove our first automorphy lifting theorem for $l$-adic Galois representations. The argument follows similar lines to \cite[\S 6]{Tho16}. The main difference is that by making use of Lemma \ref{lem_combined_dimension_bound} we can make do under weaker assumptions. Especially, we do not need an a priori bound on the dimension of the locus of reducible deformations inside $R_{\cS_1}^\text{univ}$. 

Let us take up assumptions (1)--(3) of \S \ref{subsec_set-up_for_R_equals_T}. Thus we have a CM field $L$, a unitary group $G$, and a set $S = T = S_l \cup R \cup S(B) \cup S_a$ of finite places of $L^+$ split in $L$. We have an open compact subgroup $U \subset G(\bbA_{L^+}^\infty)$, a maximal ideal $\ffrm$ of the Hecke algebra $\bbT^S(U(\frl^\infty), \cO)$ of residue field $k$, a deformation problem
\[ \cS_1 = (  L / L^+, T, \widetilde{T}, \Lambda, \overline{r}_\ffrm, \epsilon^{1-n} \delta_{L / L^+}^{n}, \{ R_v^\triangle \}_{v \in S_l} \cup \{ R_v^{1} \}_{v \in R} \cup \{ R_v^\text{St} \}_{v \in S(B) } \cup \{ R_v^\square \}_{v \in S_a} ), \]
and a diagram of $\Lambda$-algebras
\[ \xymatrix@1{R^\text{univ}_{\cS_1}  & P_{\cS_1} \ar[l]\ar[r] &\bbT^{T}_{1}(U(\frl^\infty), \cO)_{\ffrm}. } \]
We can now state the main theorem of this section. 
\begin{theorem}\label{thm_main_argument}
	Let $r : G_{L^+, S} \to \cG_n(\cO)$ be a lifting of $\overline{r}_\ffrm$ of 
	type $\cS_1$ such that $r|_{G_{L, S}}$ is ordinary of weight $\lambda$, for 
	some $\lambda \in (\bbZ_+^n)^{\Hom(L, \overline{\bbQ}_l)}$. 
	Suppose that the following conditions hold: 
	\begin{enumerate}
	\item $\overline{r}_\ffrm|_{G_{L, S}}$ is primitive. 
	\item $\zeta_l \not\in L$, $\overline{r}_\ffrm|_{G_{L^+(\zeta_l)}}$ is Schur, and $\overline{r}_\ffrm(G_{L, S})$ does not have a quotient of order $l$.
	\item $l > 3$ and $l \nmid n$.
	\item Let $d_{L, 0}, d_{L, l}$ be as defined in \S \ref{subsec_dimension_bounds}. Then $d_{L, 0} > |R|n(n+1) + 2$ and $d_{L, l} > \sup(|R|n(n+1) + 2, n(n-1)/2 + 1)$.
	\end{enumerate}
	Then $r|_{G_{L, S}}$ is automorphic of weight $\lambda$. 
\end{theorem}
We assume the hypotheses of Theorem \ref{thm_main_argument} for the rest of \S 
\ref{sec_mainargument}. Note that our assumption on $d_{L,0}$ implies that 
$S(B)$ is non-empty. In particular, the conclusion of Theorem 
\ref{thm_main_argument} implies, by local--global compatibility at a place in 
$S(B)$, that $r|_{G_{L,S}}$ is absolutely irreducible. 

Let us say that a soluble CM extension $M/L$ is good if it is linearly disjoint 
from the extension of $L(\zeta_l)$ cut out by 
$\overline{r}_\ffrm|_{G_{L(\zeta_l)}}$ and every prime above $S_l \cup S_a \cup 
R$ splits in $M$. (Primes of $S(B)$ are not required to split, and indeed this 
possibility plays an important role in the proof, cf.~\cite[Proposition 
6.2]{Tho16}.)
\begin{lemma}\label{lem_primitive_after_good_base_change}
	Let $M / L$ be a good extension. Then $\overline{r}_\ffrm|_{G_M}$ is primitive, $\zeta_l \not\in M$, $\overline{r}_\ffrm|_{G_{M^+(\zeta_l)}}$ is Schur, and $\overline{r}_\ffrm(G_M)$ does not have a quotient of order $l$.
\end{lemma}
\begin{proof}
We take each property in turn. If $M / L$ is good then $\overline{r}_\ffrm(G_M) = \overline{r}_\ffrm(G_L)$, so this image does not have a quotient of order $l$. Since $M / L$ is linearly disjoint $L(\zeta_l)$, we have $\zeta_l \not\in M$. 

To say that $\overline{r}_\ffrm|_{G_{M^+(\zeta_l)}}$ is Schur is to say that $L 
\not\subset M^+(\zeta_l)$ and, if $\overline{\rho}, \overline{\rho}'$ are two 
Jordan--H\"older factors of $\overline{r}_\ffrm|_{G_M}$, then 
$\overline{\rho}^{c} \cong \overline{\rho}^\vee \otimes \epsilon^{1-n}$ and 
$\overline{\rho} \not\cong \overline{\rho}'$. Since $\overline{r}_\ffrm(G_M) = 
\overline{r}_\ffrm(G_L)$, the latter property is not disturbed. We show that 
under our assumptions we in fact have $M \not\subset M^+(\zeta_l)$. It suffices 
to show that the two extensions $M / M^+$ and $M^+(\zeta_l) / M^+$ are linearly 
disjoint. Since these extensions arise from the linearly disjoint extensions $L 
/ 
L^+$ and $L^+(\zeta_l) / L^+$ by compositum with $M^+$, it is enough to check 
that the extensions $M^+ / L^+$ and $L(\zeta_l) / L^+$ are linearly disjoint, 
or even that $M^+ \cap L(\zeta_l) = L^+$. This follows from the stronger 
assertion that $M \cap L(\zeta_l) = L$. 

Finally the condition that $\overline{r}_\ffrm|_{G_M}$ is primitive depends 
only on the group $\overline{r}_\ffrm(G_M) = \overline{r}_\ffrm(G_L)$, so it is 
inherited from the corresponding condition for $\overline{r}_\ffrm$.
\end{proof}
If $M/L$ is a good extension and $X$ (resp. $\widetilde{X}$) is a set of places of $L^+$ (resp. $L$), then we write $X_M$ (resp $\widetilde{X}_M$) for the set of places of $M^+$ lying above a place of $X$ (resp. places of $M$ lying above a place of $\widetilde{X}$). We write $\Lambda_M = \widehat{\otimes}_{v \in S_{l, M}} \cO \llbracket I_{M_\wv}^{ab}(l)^n \rrbracket$ for the Iwasawa algebra of $M$ (so $\Lambda_L = \Lambda$). There is a natural surjective homomorphism $\Lambda_M \to \Lambda$. We can define a deformation problem
\[ \cS_{1, M} = (  M / M^+, T_M, \widetilde{T}_M, \Lambda_M, \overline{r}_\ffrm|_{G_{M^+}}, \epsilon^{1-n} \delta_{M / M^+}^{n}, \{ R_v^\triangle \}_{v \in S_{l, M}} \cup \{ R_v^{1} \}_{v \in R_M} \cup \{ R_v^\text{St} \}_{v \in S(B)_M } \cup \{ R_v^\square \}_{v \in S_{a, M}} )\]
On the automorphic side, we can define an open compact subgroup $U_M \subset G(\bbA_{M^+}^\infty)$ with the property that $U_L = U$; see \cite[\S 4.5]{Tho16} for details. With this choice, it is possible to define a maximal ideal $\ffrm_M \subset \bbT^{T_M}(U_M(\frl^\infty), \cO)$ with the property that $\overline{r}_{\ffrm_M} = \overline{r}_\ffrm|_{G_{M^+}}$, and then there exists a commutative diagram of $\Lambda_M$-algebras
\[ \xymatrix{ R_{\cS_1}^\text{univ} & \ar[l] P_{\cS_1} \ar[r] & \bbT^T_1(U(\frl^\infty), \cO)_\ffrm \\ 
R_{\cS_{1, M}}^\text{univ} \ar[u] & \ar[l] P_{\cS_{1, M}} \ar[u] \ar[r] & \bbT^{T_M}_1(U_M(\frl^\infty), \cO)_{\ffrm_M} \ar[u]
 } \]
(see \cite[Proposition 4.18]{Tho16}). Let $J_M = J_{\cS_{1, M}} P_{\cS_1}$. 
Then we have an inclusion $J_{\cS_{1, M}} P_{\cS_1} \subset J_{\cS_1}$. More 
generally, if $M_1 / M_0 / L$ is a tower of good extensions of $L$, then 
$J_{M_1} \subset J_{M_0}$. We say that a prime ideal $\frp \subset 
R_{\cS_1}^\text{univ}$ is \emph{potentially pro-automorphic} if there exists a 
good extension $M/L$ such that $J_MR_{\cS_1}^\text{univ} \subset \frp$.
\begin{proposition}\label{prop_pro-automorphy} Let $\frp \subset R_{\cS_1}^\text{univ}$ be a prime of dimension 1 and characteristic $l$ which is potentially pro-automorphic and generic.  Suppose further that for each $v \in R$, the restriction $r_\frp|_{G_{L_\wv}}$ is trivial. Then every minimal prime $Q \subset \frp$ is potentially pro-automorphic.
\end{proposition}
\begin{proof} The proof is the same as the proof of \cite[Proposition 6.2]{Tho16}. We need only replace the reference there to \cite[Corollary 4.20]{Tho16} to Theorem \ref{thm_application_of_Taylor-Wiles} here.
\end{proof}
\begin{proof}[Proof of Theorem \ref{thm_main_argument}]
For any good extension $M / L$, the ring $R_{\cS_1}^\text{univ} / (\lambda, J_M)$ is a finite $\Lambda$-algebra. Indeed, $R_{\cS_{1, M}}^\text{univ} / (\lambda, J_{\cS_{1, M}})$ is a finite $\Lambda_M$-algebra, and we can appeal to \cite[Lemma 4.16]{Tho16} and \cite[Proposition 4.17]{Tho16}. It follows from Lemma \ref{lem_combined_dimension_bound} that any quotient of $R^\text{univ}_{\cS_1}/(\lambda, J_M)$ of dimension at least $1 + \sup(n[L^+ : \bbQ] - d_{L, 0}, n[L^+ : \bbQ] - d_{L, l})$ contains a generic, potentially pro-automorphic prime $\frp$ of dimension 1 and characteristic $l$.

Fix a choice of lifting $r^\text{univ}_{\cS_1} : G_{L^+, S} \to 
\cG_n(R_{\cS_1}^\text{univ})$ representing the universal deformation. This 
induces for each $v \in R$ a homomorphism $R_v^1 \rightarrow 
R^\text{univ}_{\cS_1}$, and we let $J_R$ denote the ideal generated by the 
images of $\ffrm_{R_v^1}, v \in R$. The ideal $J_R$ is independent of the 
choice of lifting, and for any quotient $R_{\cS_1}^\text{univ}/I$ of 
characteristic $l$, we have $\dim R_{\cS_1}^\text{univ}/(J_R, I) \geq \dim 
R_{\cS_1}^\text{univ}/I - |R|n^2$, by \cite[Theorem 15.1]{Mat07} (note that 
$R_v^1/(\lambda)$ has dimension $n^2$ \cite[Proposition 3.15]{Tho16}). It 
follows 
that there exists a generic prime $\frp \subset 
R^\text{univ}_{\cS_1}/(J_R,J_L)$ of dimension 1 and characteristic $l$, since 
$\dim R^\text{univ}_{\cS_1}/(J_R, J_L) \geq n[L^+ : \bbQ] - n^2 |R| > 
\sup(n[L^+ : \bbQ] - d_{L, 0}, n[L^+ : \bbQ] - d_{L, l})$ (here we are using 
assumption (4) in the statement of Theorem \ref{thm_main_argument}). By 
Proposition \ref{prop_pro-automorphy}, any minimal prime $Q \subset \frp$ of 
$R^\text{univ}_{\cS_1}$ is potentially pro-automorphic. 

We now consider the partition of the set of minimal primes of $R^\text{univ}_{\cS_1}$ into two sets $\cC_1, \cC_2$, consisting of those primes which respectively are and are not potentially pro-automorphic. We have shown that $\cC_1$ is non-empty. We claim that $\cC_2$ is empty. Otherwise, \cite[Lemma 3.21]{Tho16} implies the existence of primes $Q_1 \in \cC_1, Q_2 \in \cC_2$ such that 
\begin{equation*} \dim R^\text{univ}_{\cS_1}/(Q_1, Q_2) \geq n[L^+ : \bbQ] - |R|n - 2,
\end{equation*}
and hence 
\begin{equation*} \dim R^\text{univ}_{\cS_1}/(Q_1, Q_2, J_R) \geq n[L^+ : \bbQ] - |R|n - |R|n^2 - 2 = n[L^+ : \bbQ] - |R|n(n+1) - 2.
\end{equation*}
 The ring $R^\text{univ}_{\cS_1}/Q_1$ (and hence each of its quotients) is finite over $\Lambda_L$. Applying Lemma \ref{lem_combined_dimension_bound} and assumption (4) once again, we see that  $R^\text{univ}_{\cS_1}/(Q_1, Q_2, J_R)$ contains a generic, potentially pro-automorphic prime $\frp'$ of dimension 1 and characteristic $l$. Applying Proposition \ref{prop_pro-automorphy} to $\frp'$, we deduce that $Q_2$ is potentially pro-automorphic, a contradiction. 
 
 Now let $r : G_{L^+, S} \rightarrow \cG_n(\cO)$ be a lifting of $\overline{r}_\ffrm$ which is ordinary of weight $\lambda$ and of type $\cS_1$, as in the statement of the theorem. This induces a homomorphism $R^\text{univ}_{\cS_1} \rightarrow \cO$. Let $Q$ be a minimal prime contained inside the kernel of this homomorphism. Then there is a good extension $M/L$ such that $J_M \subset Q$, and so the induced homomorphism $R^\text{univ}_{\cS_{1, M}} \to \cO$ kills $J_{\cS_{1, M}}$, and the map $P_{\cS_{1, M}} \rightarrow \cO$ induced by $r|_{G_{M, S_M}}$ factors through $\bbT^{T_M}_{1}(U_M(\frl^\infty),\cO)_{\ffrm_M}$. Using \cite[Lemma 2.6.4]{Ger09}, \cite[Proposition 3.3.2]{Clo08}, and \cite[Lemma 2.7]{Tho16} (respectively a classicality statement in Hida theory, base change for the unitary group $G_{M^+}$, and soluble descent for $\GL_n(\bbA_M)$) we deduce that the representation $r|_{G_L}$ is automorphic of weight $\lambda$. 
\end{proof}
The following corollary was established in the course of the above proof.
\begin{corollary}\label{cor_finiteness}
	With hypotheses on $\overline{r}_\ffrm$ as in Theorem \ref{thm_main_argument}, $R_{\cS_1}$ is a finite $\Lambda$-algebra.
\end{corollary}

\section{The end}\label{sec_main_theorem}

We are now in a position to state and prove the main theorem of this paper. 
\begin{theorem}\label{thm_new_main_theorem} Let $F$ be an imaginary CM number field with maximal totally real subfield $F^+$, and let $n \geq 2$ be an integer. Let $l$ be a prime, and suppose that $\rho : G_F \rightarrow \mathrm{GL}_n(\overline{\bbQ}_l)$ is a continuous semisimple representation satisfying the following hypotheses.
\begin{enumerate} \item $\rho^c \cong \rho^\vee \epsilon^{1-n}$.
\item $\rho$ is ramified at only finitely many places.
\item  $\rho$ is ordinary of weight $\lambda$, for some $\lambda \in (\bbZ_+^n)^{\Hom(F, \barQl)}$.
\item There is an isomorphism $\overline{\rho}^\text{ss} \cong 
\overline{\rho}_1 \oplus \dots \oplus \overline{\rho}_\num$, where each 
$\overline{\rho}_i$ is absolutely irreducible and satisfies 
$\overline{\rho}_i^c \cong \overline{\rho}_i^\vee \epsilon^{1-n}$, and 
$\overline{\rho}_i \not\cong \overline{\rho}_j$ if $i \neq j$. 
\item There exists a finite place $\wv_0$ of $F$, prime to $l$, such that 
$\rho|_{G_{F_{\wv_0}}}^\text{ss} \cong \oplus_{i=1}^n \psi \epsilon^{n-i}$ for 
some unramified character $\psi : G_{F_{\wv_0}} \rightarrow 
\overline{\bbQ}_l^\times$.
\item There exists a RACSDC representation $\pi$ of $\mathrm{GL}_n(\bbA_F)$ and $\iota : \barQl \to \bbC$ such that: \begin{enumerate} \item $\pi$ is $\iota$-ordinary. 
\item $\overline{r_{ \iota}(\pi)}^\text{ss} \cong \overline{\rho}^\text{ss}$.
\item $\pi_{\wv_0}$ is an unramified twist of the Steinberg representation.
\end{enumerate} 
\item\label{ass:prim} $F(\zeta_l)$ is not contained in $\overline{F}^{\ker \ad 
(\overline{\rho}^\text{ss})}$ and $F$ is not contained in $F^+(\zeta_l)$. For 
each $1 \leq i, j \leq \num$, 
 $\overline{\rho}_i|_{G_{F(\zeta_l)}}$ is absolutely 
irreducible and $\overline{\rho}_i|_{G_{F(\zeta_l)}} \not\cong 
\overline{\rho}_j|_{G_{F(\zeta_l)}}$ if $i \neq j$. Moreover, 
$\overline{\rho}^\text{ss}$ is primitive and $\overline{\rho}^\text{ss}(G_{F})$ 
has no quotient of order $l$. 
\item $l > 3$ and $l \nmid n$.
\end{enumerate}
	Then $\rho$ is automorphic: there exists an $\iota$-ordinary RACSDC automorphic representation $\Pi$ of $\GL_n(\bbA_F)$ such that $r_\iota(\Pi) \cong \rho$. 
\end{theorem}
\begin{proof}
The proof is similar to, and even simpler than, the proof of \cite[Theorem 7.1]{Tho16}. We can find a coefficient field $K \subset \overline{\bbQ}_l$ and, after replacing $\rho$ by a conjugate, assume that $\rho = r|_{G_F}$, where $r : G_{F^+} \to \cG_n(\cO)$ is a homomorphism such that $\overline{r}$ is Schur and $\nu \circ r = \epsilon^{1-n} \delta_{F / F^+}^n$. Similarly we can assume the existence of a model $\rho' : G_{F} \to \GL_n(\cO)$ for $r_\iota(\pi)$ which extends to a homomorphism $r' : G_{F^+} \to \cG_n(\cO)$ such that $\nu \circ r' = \epsilon^{1-n} \delta_{F / F^+}^n$ and such that $\overline{r}' = \overline{r}$. In fact, $\overline{r}|_{G_{F^+(\zeta_l)}}$ is Schur. Indeed, making use of assumption \ref{ass:prim}, we just need to check that  $\overline{r}(G_{F^+(\zeta_l)})$ 
meets both components of $\cG_n$. This is our assumption that $F \not\subset F^+(\zeta_l)$.

After making a preliminary soluble base change, we can assume that the following further conditions are satisfied:
\begin{enumerate}
\item $F / F^+$ is everywhere unramified, and each place of $F^+$ dividing $l$ or above which $\rho$ or $\pi$ is ramified splits in $F$.
\item The place $\wv_0$ is split over $F^+$. We write $v_0 = \wv_0|_{F^+}$.
\item For each place $w$ of $F$ at which $\rho$ or $\pi$ is ramified, or which divides $l$, the representation $\overline{r}|_{G_{F_w}}$ is trivial. 
\item For each prime-to-$l$ place $w$ of $F$ at which $\rho$ or $\pi$ is ramified, we have $q_w \equiv 1 \text{ mod }l$, and if $l^N||(q_w - 1)$ then $l^N > n$. Moreover, $\rho|_{G_{F_w}}$ and $\rho'|_{G_{F_w}}$ are unipotently ramified.
\end{enumerate}
We can find a finite set $\widetilde{X}_0$ of finite places of $F$ satisfying the following conditions:
\begin{itemize}
\item $\widetilde{X}_0$ does not contain any place at which $\rho$ or $\pi$ is ramified, or any place dividing $l$.
\item Let $E / F(\zeta_l)$ denote the extension cut out by $\overline{\rho}|_{G_{F(\zeta_l)}}$. Then for any Galois subextension $E / E'/ F$ with $\Gal(E'/F)$ simple, there exists a place $w \in \widetilde{X}_0$ which does not split in $E'$. 
\item $\widetilde{X}_0$ contains an absolutely unramified place $\wv_1$ such that $\overline{\rho}(\Frob_{\wv_1})$ is scalar and $q_{\wv_1} \not\equiv 1 \text{ mod }l$.
\end{itemize}
Let $X_0$ denote the set of places of $F^+$ lying below a place of $\widetilde{X}_0$. We write $v_1 = \wv_1|_{F^+}$. If $L / F$ is any Galois CM extension which is $\widetilde{X}_0$-split, then the analogue of assumption (7) of the theorem (where $F$ is replaced by $L$ and $\overline{\rho}$ by $\overline{\rho}|_{G_L}$) is satisfied.

If $L^+ / F^+$ is a Galois, totally real, $X_0$-split extension and $L = L^+ \cdot F$, then $L / F$ is Galois CM and $\widetilde{X}_0$-split. We claim that we can find a soluble, totally real, $X_0$-split extension such that the hypotheses of Theorem \ref{thm_main_argument} are satisfied for $r|_{G_{L^+}}$. This will complete the proof: we deduce from Theorem \ref{thm_main_argument} that $r|_{G_L} = \rho|_{G_L}$ is automorphic of weight $\lambda$ and irreducible. The automorphy of $\rho$ then follows by soluble descent.

Let $\widetilde{Y}_0$ denote the set of places $\wv \neq \wv_0, \wv_0^c$ of $F$ 
dividing $l$ or at which $\rho$ or $\pi$ is ramified, and let $Y_0$ denote the 
set of places of $F^+$ lying below a place of $\widetilde{Y}_0$. Then every 
place of $Y_0$ splits in $F$ and $Y_0 \cap ( X_0 \cup \{ v_0 \} ) = \emptyset$. 
For any odd integer $\delta \geq 1$, we can find a cyclic totally real 
extension $M_0 / F^+$ of degree $\delta$ which is $X_0 \cup \{ v_0 \}$-split 
and such that each place $v \in Y_0$ of $F^+$ is totally inert in $M_0$.

Let $M_1$ be a totally real quadratic extension of $F^+$ which is $X_0 \cup \{ 
v_0 \} \cup Y_0$-split. We will take $L = F \cdot M_0 \cdot M_1$. We claim that 
for a suitable choice of odd integer $\delta$, this $L$ will suffice for the 
application of Theorem \ref{thm_main_argument}. More precisely, we will apply 
Theorem \ref{thm_main_argument} to the homomorphism $r|_{G_{L^+}}$ with the 
following data:
\begin{itemize}
\item $S_l$ is the set of $l$-adic places of $L^+$.
\item $S(B)$ is the set of places of $L^+$ lying above $v_0$.
\item $R$ is the set of prime-to-$l$ places of $L^+$ lying above a place of $Y_0$.
\item $S_a$ is the set of places of $L^+$ lying above $v_1$. 
\end{itemize}
Let $\Delta$ denote the Galois group of the maximal abelian pro-$l$ extension of $L$ unramified outside $l$, and let $\Delta_0$ denote the Galois group of the maximal abelian pro-$l$ extension of $L$ which is unramified outside $l$ and which is $S(B)$-split. Then $\Delta_0$ is naturally a quotient of $\Delta$. Let us define
\[ d_0 = \dim_{\bbQ_l} \ker(\Delta[1/l] \to \Delta_0[1/l])^{c = -1} \]
and
\[ d_l = \inf_{v \in S_l} [L^+_v : \bbQ_l]. \]
In order to complete the proof, we must show that the odd integer $\delta \geq 
1$ can be chosen so that the inequalities 
\[ d_0 > |R| n(n+1)/2 + 2 \]
and
\[ d_l > \sup(|R|n(n+1)/2 + 2, n(n-1)/2+1) \]
are simultaneously satisfied. It follows from \cite[Proposition 19]{Mai02} that 
$d_0 = 2\delta$ (see the proof of \cite[Theorem 7.1]{Tho16} for more details). 
If $\delta$ is prime to the absolute residue degrees of all the places of 
$Y_0$, then we will have $|R| \leq 2 |Y_0|$ and $d_l \geq \delta$. We will 
therefore be done if we can choose $\delta$ to satisfy
\[ 2\delta > 2 |Y_0| n(n+1)/2 + 2 \]
and
\[ \delta > \sup(2|Y_0|n(n+1)/2+2, n(n-1)/2 + 1). \]
This is clearly possible, and concludes the proof.
\end{proof}
We can use the same idea to prove the following finiteness result (compare 
\cite[Theorem 10.2]{Tho11}), which plays a crucial role in some of the level 
raising arguments in \cite{new-tho-symm}.
\begin{theorem}
	Let $F$ be a CM field, let $l$ be a prime, and let $\iota : \overline{\bbQ}_l \to \bbC$ be an isomorphism. Let $S$ be a finite set of finite places of $F^+$, containing the $l$-adic ones, and suppose that each place of $S$ splits in $F$. Choose for each $v \in S$ a place $\wv$ of $F$ lying above $v$, and let $\widetilde{S}$ denote the set of the $\wv$.
	
	Let $\pi$ be a RACSDC automorphic representation of $\GL_n(\bbA_F)$, and let $K / \bbQ_l$ be a coefficient field such that $r_{\iota}(\pi)$ can be chosen to take values in $\GL_n(\cO)$, and extend it to a homomorphism $r : G_{F^+} \to \cG_n(\cO)$ such that $\nu \circ r = \epsilon^{1-n} \delta_{F / F^+}^n$.	Suppose that the following conditions are satisfied:
	\begin{enumerate}
		\item $\pi$ is $\iota$-ordinary. For each place $v \in S_l$, $\overline{r}|_{G_{F_\wv}}$ is the trivial representation. 
		\item $\pi$ is unramified outside $S$.
		\item There exists a place $v_0 \nmid l$ of $S$ such that $\pi_{\wv_0}$ is an unramified twist of the Steinberg representation. Moreover, $q_{v_0} \equiv 1 \text{ mod }l$ and $\overline{r}|_{G_{F_{\wv_0}}}$ is the trivial representation.
		\item Let $\overline{\rho} = \overline{r}|_{G_{F, S}}$. There is an 
		isomorphism $\overline{\rho}^\text{ss} \cong \overline{\rho}_1 \oplus 
		\dots \oplus \overline{\rho}_\num$, where each $\overline{\rho}_i$ is 
		absolutely irreducible and satisfies $\overline{\rho}_i^c \cong 
		\overline{\rho}_i^\vee \epsilon^{1-n}$. (In particular, $\overline{r}$ 
		is Schur and therefore $\overline{\rho}$ is semisimple.)
		\item $F(\zeta_l)$ is not contained in $\overline{F}^{\ker \ad 
			(\overline{\rho}^\text{ss})}$ and $F$ is not contained in $F^+(\zeta_l)$. For 
		each $1 \leq i, j \leq \num$, 
		$\overline{\rho}_i|_{G_{F(\zeta_l)}}$ is absolutely 
		irreducible and $\overline{\rho}_i|_{G_{F(\zeta_l)}} \not\cong 
		\overline{\rho}_j|_{G_{F(\zeta_l)}}$ if $i \neq j$. Moreover, 
		$\overline{\rho}$ is primitive and 
		$\overline{\rho}(G_{F})$ has no quotient of order $l$.
		\item $l > 3$ and $l \nmid n$.
	\end{enumerate}
	Define the global deformation problem
	\[ \cS = ( F / F^+, S, \widetilde{S}, \Lambda, \overline{r}, \epsilon^{1-n} \delta^n_{F / F^+}, \{ R_v^\triangle \}_{v \in S_l} \cup \{ R_v^\square \}_{v \in S - (S_l \cup \{ v_0 \})} \cup \{ R_{v_0}^{St}\} ). \]
	Then $R_{\cS}^\text{univ}$ is a finite $\Lambda$-algebra. 
\end{theorem}
\begin{proof}
	Let $R = S - S_l \cup \{ v_0 \}$, $S(B) = \{ v_0 \}$. If $L / F$ is a CM extension, let $S_L$ denote the set of places of $L^+$ above $S$, $\widetilde{S}_L$ the set of places of $L$ above $\widetilde{S}$, and define $R_L$, $\widetilde{R}_L$ etc.\ similarly. If $L$ satisfies the following conditions:
	\begin{itemize}
		\item $\overline{r}|_{L^+}$ is Schur;
		\item For each place $v \in R_L$, $q_v \equiv 1 \text{ mod }l$ and $\overline{r}|_{G_{L_\wv}}$ is the trivial representation;
		\item For each place $v \in R$ and each place $w | v$ of $L^+$, the induced map $R_{L_\wv}^\square \to R_{F_\wv}^\square$ factors through the quotient $R_{L_\wv}^\square\to R_{L_\wv}^1$,
	\end{itemize}
	then we can define the global deformation problem
	\[ \cS_L = ( L / L^+, S_L, \widetilde{S}_L, \Lambda_L, \overline{r}|_{G_{L^+}}, \epsilon^{1-n} \delta^n_{L / L^+}, \{ R_v^\triangle \}_{v \in S_{l, L}} \cup \{ R_v^1 \}_{v \in R_L} \cup \{ R_{v}^{St}\}_{v \in S(B)_L} ), \]
	and restriction from $F^+$ to $L^+$ determines a finite morphism $R_{\cS_L}^\text{univ} \to R_{\cS}^\text{univ}$. If moreover $L$ satisfies the conditions of Theorem \ref{thm_main_argument}, then Corollary \ref{cor_finiteness} will imply that $R_{\cS_L}^\text{univ}$ is a finite $\Lambda_L$-algebra, hence that $R_{\cS}^\text{univ}$ is a finite $\Lambda$-algebra. Such an extension $L / F$ can be constructed in exactly the same manner as in the proof of Theorem \ref{thm_new_main_theorem}.
\end{proof}

\bibliographystyle{alpha}
\bibliography{ModLiftBib}
\end{document}